\documentclass[11pt]{article}
\usepackage{latexsym}
\usepackage{amsmath}
\usepackage{amsthm,color}
\usepackage{amssymb}
\usepackage{mathrsfs}
\usepackage{picinpar,graphicx}
\usepackage[colorlinks,  linkcolor=blue,  anchorcolor=blue, citecolor=blue]{hyperref}

\usepackage{lscape}
\newtheorem{theorem}{\indent Theorem}[section]
\newtheorem{proposition}[theorem]{\indent Proposition}

\newtheorem{lemma}[theorem]{\indent Lemma}
\newtheorem{remark}{\indent Remark}[section]

\begin{document}
\renewcommand{\baselinestretch}{1.3}


\begin{center}
    {\large \bf Existence  of normalized solutions for  fractional  coupled Hartree-Fock type system}
\vspace{0.5cm}\\{\sc Meng  Li$^*$}  
\end{center}


\renewcommand{\theequation}{\arabic{section}.\arabic{equation}}
\numberwithin{equation}{section}


\begin{abstract}
In this paper, we consider the existence of solutions for the following fractional  coupled Hartree-Fock type system
\begin{align*}
\left\{\begin{aligned}
&(-\Delta)^s u+V_1(x)u+\lambda_1u=\mu_1(I_{\alpha}\star |u|^p)|u|^{p-2}u+\beta(I_{\alpha}\star |v|^r)|u|^{r-2}u\\
&(-\Delta)^s v+V_2(x)v+\lambda_2v=\mu_2(I_{\alpha}\star |v|^q)|v|^{q-2}v+\beta(I_{\alpha}\star |u|^r)|v|^{r-2}v
\end{aligned}
\right.~\quad x\in\mathbb{R}^N,
\end{align*}
under the constraint
\begin{align*}
\int_{\mathbb{R}^N}|u|^2=a^2,~\int_{\mathbb{R}^N}|v|^2=b^2.
\end{align*}
where $s\in(0,1),~N\ge3,~\mu_1>0,~\mu_2>0,~\beta>0,~\alpha\in(0,N),~1+\frac{\alpha}{N}<p,~q,~r<\frac{N+\alpha}{N-2s}$ and $I_{\alpha}(x)=|x|^{\alpha-N}$. Under some restrictions of $N,\alpha,p,q$ and $r$, we give the positivity of normalized solutions for $p,q,r\le 1+\frac{\alpha+2s}{N}$.
\end{abstract}

\vspace{-1 cm}

\footnote[0]{ \hspace*{-7.4mm}
$^{*}$ Corresponding author.
E-mails: mengl@hust.edu.cn}

\section{Introduction}
In this paper, we consider the existence of solutions for the following fractional  coupled Hartree-Fock type system
\begin{align}\label{system}
\left\{\begin{aligned}
&(-\Delta)^s u+V_1(x)u+\lambda_1u=\mu_1(I_{\alpha}\star |u|^p)|u|^{p-2}u+\beta(I_{\alpha}\star |v|^r)|u|^{r-2}u\\
&(-\Delta)^s v+V_2(x)v+\lambda_2v=\mu_2(I_{\alpha}\star |v|^q)|v|^{q-2}v+\beta(I_{\alpha}\star |u|^r)|v|^{r-2}v
\end{aligned}
\right.~\quad x\in\mathbb{R}^N,
\end{align}
under the constraint
\begin{align}\label{condition}
\int_{\mathbb{R}^N}|u|^2=a^2,~\int_{\mathbb{R}^N}|v|^2=b^2.
\end{align}
where $s\in(0,1),~N\ge3,~\mu_1>0,~\mu_2>0,~\beta>0,~\alpha\in(0,N),~1+\frac{\alpha}{N}<p,~q,~r<\frac{N+\alpha}{N-2s}$ and $I_{\alpha}(x)=|x|^{\alpha-N}$.

When studying  normalized solutions to the fractional equation
\begin{align*}
(-\Delta)^s u+\lambda u=(I_{\alpha}\star|u|^p)|u|^{p-2}u,~x\in\mathbb{R}^N,
\end{align*}
the $L^2$-critical exponent $p=1+\frac{\alpha+2s}{N}$, the Hardy-Littlewood-Sobolev upper critical exponent $\bar{p}=\frac{\alpha+N}{N-2s}$ and lower critical exponent $\b{p}=1+\frac{\alpha}{N}$ play an important role.

For $s=1$, this nonlocal type problem was considered in the basic
quantum chemistry model of small number of electrons interacting with static nucleii which can be approximated by Hartree or Hartree-Fock minimization problems(see \cite{LL05,LS77,L87}). When $N=3,\alpha=2$ and $p=q=r=2$, the system \eqref{system}($s=1$) arises in multi-component Bose-Einstein condensates \cite{AEMWC95} and nonlinear optics \cite{M87}. Actually, considering the generalized time-dependent Schr\"odinger system
\begin{align}\label{1system}
\left\{\begin{array}{l}
-i \frac{\partial \Phi_1}{\partial t}+V_1(x) \Phi_1=\frac{\hbar^2}{2 m} \Delta \Phi_1+\mu_1(I_\alpha\star|\Phi_1|^2) \Phi_1+\beta(I_\alpha \star|\Phi_2|^2) \Phi_1, \\
-i \frac{\partial \Phi_2}{\partial t}+V_2(x) \Phi_2=\frac{\hbar^2}{2 m} \Delta \Phi_2+\mu_2(I_\alpha \star|\Phi_2|^2) \Phi_2+\beta(I_\alpha \star|\Phi_1|^2) \Phi_2, \\
\Phi_j=\Phi_j(x, t) \in \mathbb{C}, \Phi_j(x, t) \to 0, \quad \text {~as~}|x| \to \infty, t>0, j=1,2,
\end{array}\right.
\end{align}
where $i$ is the imaginary unit, $m$ is the mass of the particles, $\hbar$ is the Plank constant, $\mu_1,\mu_2>0$, and $\beta \neq 0$ is a coupling constant which describes the scattering length of the attractive or repulsive interaction, $V_1(x)$ and $V_2(x)$ are the external potentials. Physically, the solution $\Phi_i$ denotes the $i$-th component of the beam in Kerr-like photorefractive media. The positive constants $\mu_1, \mu_2$ indicate the self-focusing strength in the component of the beam, and the coupling constant $\beta$ measures the interaction between the two components of the beam. The sign of $\beta$ determines whether the interactions of states are repulsive or attractive. Any solution of system \eqref{1system} subjects to conservation of mass, that is the following two norms:
$$
\int_{\mathbb{R}^N}|\Phi_1(t, x)|^2dx \text{~and~} \int_{\mathbb{R}^N}|\Phi_2(t, x)|^2dx
$$
are independent of $t\in\mathbb{R}$. Moreover, the $L^2$-norms $|\Phi_1(t, \cdot)|_2$ and $|\Phi_2(t, \cdot)|_2$ have important physical significance, for example, in Bose-Einstein condensates, $|\Phi_1(t, \cdot)|_2$ and $|\Phi_2(t, \cdot)|_2$ represent the number of particles of each component; in nonlinear optics framework, $|\Phi_1(t, \cdot)|_2$ and $|\Phi_2(t, \cdot)|_2$ represent the power supply. Therefore, it is natural to consider the masses as preserved, and the solution of \eqref{1system} with prescribed mass is called normalized solution.
To study the solitary wave solution of \eqref{1system}, we set $\Phi_1(x, t)=e^{i \lambda_1 t} u(x)$ and $\Phi_2(x, t)=e^{i \lambda_2 t} v(x)$. Then, system \eqref{1system} is reduced to a general elliptic system \eqref{system} with $s=1$.

Condition \eqref{condition}  is called as the normalization condition, which imposes a normalization on the $L^2$-masses of $u$ and $v$. The solutions to the Schr\"odiger system \eqref{system} under the constraint \eqref{condition} are normalized solutions. In order to obtain the solution to the fractional Schr\"odiger system \eqref{system}  satisfying the normalization condition \eqref{condition}, one need to consider
the critical point of the functional $E_{\mu_1,\mu_2,\beta}(u,v)$ on  $S_{a}\times S_{b}$(see \eqref{efct} and \eqref{l2sc}). And then, $\lambda_1$ and $\lambda_2$ appear as Lagrange multipliers with respect to the mass constraint, which cannot be determined a priori, but are part of the unknown. Some literature called this problem as fixed mass problem.

Recently, the normalized solutions of nonlinear
Schr\"{o}dinger equations and systems has attracted many researchers, see more details \cite{BC13,BJ18,BJN16,B20,BZZ20,GJ18,J2008,J20,LZ21,S20,SN20,WW21}. In particular, for $s=1$, Wang in \cite{W21} considered the system \eqref{system} with $1+\frac{\alpha+2s}{N}<p,q<\frac{N+\alpha}{N-2s}$, by min-max principle and Liouville type theorem, he gave the existence of normalized solutions. Wang and Yang in \cite{WY18} considered  \eqref{system} with $L^2-$ critical exponent, they gave the  existence and asymptotic behaviours of normalized solutions.

Up to our knowledge, there is no paper about the normalized solutions of \eqref{system} for $s\in(0,1)$. Therefore, we consider the existence of normalized solutions for \eqref{system} under the different assumptions of $p,q,r$ and give the positivity of normalized solutions under the trapping potentials.

The corresponding energy functional with \eqref{system} is
\begin{align}\label{efct}
\begin{split}
&E_{\mu_1,\mu_2,\beta}(u,v)\\
=&\frac12\int_{\mathbb{R}^N}(|(-\Delta)^{\frac{s}{2}}u|^2+|(-\Delta)^{\frac{s}{2}}v|^2)dx+\frac12\int_{\mathbb{R}^N}(V_1(x)|u|^2+V_2(x)|v|^2)dx\\
&-\frac{\mu_1}{2p}\int_{\mathbb{R}^N}(I_{\alpha}\star|u|^p)|u|^pdx-\frac{\mu_2}{2q}\int_{\mathbb{R}^N}(I_{\alpha}\star|v|^q)|v|^qdx-\frac{\beta}{r}\int_{\mathbb{R}^N}(I_{\alpha}\star|u|^r)|v|^rdx.
\end{split}
\end{align}
It is standard to check that $E_{\mu_1,\mu_2,\beta}\in C^1$ under some assumptions on $p,q$ and $r$. The critical point of $E_{\mu_1,\mu_2,\beta}$ constrained to $S_a\times S_b$ gives rise to a solution for \eqref{system}, where
\begin{align}\label{l2sc}
S_a=\{u\in H^s(\mathbb{R}^N),\|u\|_2^2=a^2\}.
\end{align}
In this paper, we mainly consider different potentials as follows:
\begin{align*}
\textbf{(V1)}:V_1(x)=V_2(x)=0,
\end{align*}
and
\begin{align*}
\textbf{(V2)}:V_{i}(x)\ge0,~V_{i}\in L_{loc}^{\infty}(\mathbb{R}^N),~ \lim_{|x|\to\infty} V_{i}(x)=\infty, \forall~i=1,2.
\end{align*}
\begin{theorem}\label{t1}
Suppose that $3\le N\le 4s,\alpha\in(0,N),p,q,r\in(1+\frac{\alpha}{N},1+\frac{\alpha+2s}{N})$ and $V_{i}(x)(i=1,2)$ satisfies \textbf{(V1)}. Let $\mu_1,\mu_2,\beta,a$ and $b$ be fixed. Then the system \eqref{system} has a solution $(u,v,\lambda_1,\lambda_2)$ with  $(u,v)\in S_a\times S_{b}$ , where $\lambda_1,\lambda_2>0$ and $u,v$ are positive and radial.
\end{theorem}

\begin{theorem}\label{t11}
Suppose that $N\ge3,\alpha\in(0,N),p,q,r\in(1+\frac{\alpha}{N},1+\frac{\alpha+2s}{N})$ and $V_{i}(x)(i=1,2)$ satisfies \textbf{(V2)}. Let $\mu_1,\mu_2,\beta,a$ and $b$ be fixed. Then the system \eqref{system} has a solution $(u,v,\lambda_1,\lambda_2)$ with  $(u,v)\in S_a\times S_{b}$, where $u,v$ are nonnegative.
\end{theorem}

\begin{theorem}\label{t12}
Let $3\le N<4s,\frac{N^2-8s^2}{4s-N}<\alpha<N$. Suppose that $p,q,r\in(\max\{2,1+\frac{\alpha}{N}\},1+\frac{\alpha+2s}{N})$ and $V_{i}(x)(i=1,2)$ satisfies \textbf{(V2)}. Let $\mu_1,\mu_2,\beta,a$ and $b$ be fixed. Furthermore, there exist $\theta\in(0,N+2s)$ and a constant $C>0$ such that
$$
\limsup_{|x|\to\infty}|x|^{-\theta}V(x)\le C.
$$
Then there exist $\eta\in(0,1)$ and a constant $\bar{C}>0$ such that $u_0,v_0\in C^{0,\eta}(\mathbb{R}^N)$ and
$$
u_0(x),v_0(x)\le \frac{\bar{C}}{|x|^{N+2s}}\quad\text{for}~|x|\ge1.
$$
In addition, $(u_0,v_0)$ is a positive  solution of the system \eqref{system} with  $(u_0,v_0)\in S_a\times S_{b}$ .
\end{theorem}

Denote that $a^{*}=\|Q\|_{2}^{2(p-1)}$, where $Q$ is the ground state solution of $(-\Delta)^{s}Q+Q-(I_{\alpha}\star|Q|^p)|Q|^{p-2}Q=0$.

\begin{theorem}\label{t2}
Suppose that $N\ge3,\alpha\in(0,N),p=q=r=1+\frac{\alpha+2s}{N}$ and $V_{i}(x)(i=1,2)$ satisfies \textbf{(V2)}.  There holds
\begin{itemize}
\item [(i)] if~$0<\mu_1a^{2(p-1)},\mu_2b^{2(p-1)}<a^{*}$ and $0<\beta<\sqrt{(a^{*}-\mu_1a^{2(p-1)})(a^{*}-\mu_2b^{2(p-1)})}$, then the system \eqref{system} has a  solution $(u,v,\lambda_1,\lambda_2)$ with  $(u,v)\in S_a\times S_{b}$ , where $u$ and $v$ are nonnegative.
\item [(ii)] if $\mu_1a^{2(p-1)}>a^{*}$ or $\mu_2b^{2(p-1)}>a^{*}$ or $\beta>\frac{(a^2+b^2)a^{*}-\mu_1a^{2p}-\mu_2b^{2p}}{2a^{p}b^{p}}$, then the system \eqref{system} has no solution.
\end{itemize}

\end{theorem}

\begin{theorem}\label{t21}
Let $3\le N<4s,\frac{N^2-8s^2}{4s-N}<\alpha<N$.
Suppose that $p=q=r=1+\frac{\alpha+2s}{N}$ and $V_{i}(x)(i=1,2)$ satisfies \textbf{(V2)}. Let  $0<\mu_1a^{2(p-1)}<a^{*}, 0<\mu_2b^{2(p-1)}<a^{*}$ and $\beta<\sqrt{(a^{*}-\mu_1a^{2(p-1)})(a^{*}-\mu_2b^{2(p-1)})}$. Furthermore, there exist $\theta\in(0,N+2s)$ and a constant $C>0$ such that
$$
\limsup_{|x|\to\infty}|x|^{-\theta}V(x)\le C.
$$
Then there exist $\eta\in(0,1)$ and a constant $\bar{C}>0$ such that $u_0,v_0\in C^{0,\eta}(\mathbb{R}^N)$ and
$$
u_0(x),v_0(x)\le \frac{\bar{C}}{|x|^{N+2s}}\quad\text{for}~|x|\ge1.
$$
In addition, $(u_0,v_0)$ is a positive solution of the system \eqref{system} with  $(u_0,v_0)\in S_a\times S_{b}$.
\end{theorem}

\begin{theorem}\label{t3}
Suppose that $N\ge 3,p,q,r\in(\max\{2,1+\frac{\alpha+2s}{N}\},\frac{N+\alpha}{N-2s}),\alpha\in(0,N)$ and $V_{i}(x)(i=1,2)$ satisfies \textbf{(V1)}. Let $\mu_1,\mu_2,a$ and $b$ be fixed and let $\beta_1>0$ be large enough. If $\beta>\beta_1$,  then the system \eqref{system} has a solution $(u,v,\lambda_1,\lambda_2)$ with $(u,v)\in S_a\times S_b$ such that $\lambda_1,\lambda_2>0$ and $u,v$ are positive and radial.
\end{theorem}

There are some difficulties to overcome in our paper. First, due to the condition \eqref{condition}, the Lagrange multipliers $\lambda_1$ and $\lambda_2$ are unknown, hence the $L^2$-convergence of critical sequence is difficult. If $V_{i}(x)(i=1,2)$ satisfies \textbf{(V2)}, then Lemma \ref{lce} solves this difficulty easily. However,  if $V_{i}(x)(i=1,2)$ satisfies \textbf{(V1)}, then Lemma \ref{lce} is no longer applicable. In this case, the sign of $\lambda_{i}(i=1,2)$ plays a key role in $L^2$-convergence of critical sequence. Therefore, we devote to showing the sign of $\lambda_{i}(i=1,2)$ by different ways. For $L^2-$subcritical case, by proving $u\not\equiv0$ and $v\not\equiv0$, combining with Liouville type lemma(see Lemma \ref{llt}), it yields the the sign of $\lambda_{i}(i=1,2)$. While for $L^2-$supercritical case, by $E'_{\mu_1,\mu_2,\beta}\to0$ and Pohozaev identity, we determine the sign of one of Lagrange multipliers,  the sign of  another Lagrange multiplier is determined by Liouville type lemma and energy comparison method.

Second, $\lim\limits_{|x|\to\infty}V_{i}(x)=\infty(i=1,2)$ brings some difficulties to obtain the positivity of the solutions for \eqref{system}. In this case, we can not obtain the positivity of nonnegative solutions directly by strong maximum principle of fractional equations. To overcome this difficulty, by giving some restrictions about $N,\alpha$ and some conditions of growth for potential functions  $V_{i}(x)(i=1,2)$, we construct the decay estimates and the H\"oder continuity of nonnegative solutions. Therefore, by the definition of fractional Laplacian operator, the the positivity of nonnegative solutions can be obtained.

Finally, compare to local terms $|u|^{p-2}u$ and $|v|^{q-2}v$, the nonlocal terms $(I_{\alpha}\star |u|^p)|u|^{p-2}u$, $(I_{\alpha}\star |v|^r)|u|^{r-2}u$, $ (I_{\alpha}\star |v|^q)|v|^{q-2}v$ and $(I_{\alpha}\star |v|^r)|u|^{r-2}u$ are more complicated. Therefore, we need some  careful analysis.

\section{Preliminaries}
In this section, we list some useful lemmas which will be used in later. For simplicity, denote that
\begin{align*}
A(u,v)=\int_{\mathbb{R}^N}(|(-\Delta)^{\frac{s}{2}}u|^2+|(-\Delta)^{\frac{s}{2}}v|^2)dx,
\end{align*}
and
$$
B(u,v,r)=\int_{\mathbb{R}^N}(I_{\alpha}\star|u|^r)|v|^r.
$$
First, we give some important inequalities.
\begin{lemma}(Fractional Gagliardo-Nirenberg inequality)
Let $s\in(0,1),~N>2s$ and $q\in(2,2^{*}_s)$. Then there exists a constant $C(N,s,q)>0$ such that
\begin{align*}
\|u\|_{q}\le C(N,s,q) \|(-\Delta)^{\frac{s}{2}}u\|_2^{\gamma_{q}}\|u\|_2^{1-\gamma_{q}}, ~\forall u\in H^{s}(\mathbb{R}^N),
\end{align*}
where $\gamma_q=\frac{N(q-2)}{2qs}$.
\end{lemma}

\begin{lemma}\label{lhls}(Hardy-Littlewood-Sobolev inequality)
Let $t,r>1$ and $\alpha\in(0,N)$ with $\frac{1}{t}+\frac{1}{r}=1+\frac{\alpha}{N},f\in L^{t}(\mathbb{R}^N)$ and $h\in L^{r}(\mathbb{R}^N)$. Then there exists a sharp constant $C(N,\alpha, t,r)>0$ such that
\begin{align*}
\bigg|\int_{\mathbb{R}^N}\int_{\mathbb{R}^N}\frac{|f(x)||h(y)|}{|x-y|^{N-\alpha}}dxdy\bigg|\le C(N,\alpha, t,r)\|f\|_{t}\|h\|_{r}.
\end{align*}
Moreover, if $t=r=\frac{2N}{N+\alpha}$, then
\begin{align*}
C(N,\alpha):=C(N,\alpha, t,r)=\pi^{\frac{N-\alpha}{2}}\frac{\Gamma(\frac{\alpha}{2})}{\Gamma(\frac{N+\alpha}{2})}\bigg\{\frac{\Gamma(\frac{N}{2})}{\Gamma(N)}\bigg\}^{-\frac{\alpha}{N}}.
\end{align*}
\end{lemma}

In the following, we give the sharp Gagliardo-Nirenberg type inequality which can be seen in \cite{FZ18,S19}.
\begin{lemma}\label{lggn}
Let $1+\frac{\alpha}{N}<p<\frac{N+\alpha}{N-2s}$. Then the best constant in the generalized Gagliardo-Nirenberg inequality
\begin{align}\label{gGN}
B(u,u,p)\le C(N,p,s,\alpha)\|(-\Delta)^{\frac{s}{2}}u\|_2^{2p\delta_{p}}\|u\|_2^{2p(1-\delta_{p})}
\end{align}
is
$$
C(N,p,s,\alpha)=\frac{2sp}{2sp-Np+N+\alpha}\bigg(\frac{2sp-Np+N+\alpha}{N(p-1)-\alpha}\bigg)^{p\delta_{p}}\|Q\|_2^{2-2p},
$$
where $Q$ is the ground state solution of the elliptic equation
\begin{align}\label{Qequation}
(-\Delta)^{s}Q+Q-(I_{\alpha}\star|Q|^p)|Q|^{p-2}Q=0.
\end{align}
In particular, if $p=1+\frac{\alpha+2s}{N}$, then $C(N,p,s,\alpha)=p\|Q\|_2^{2-2p}$.
\end{lemma}
\begin{remark}\label{rpi}
If $p=1+\frac{\alpha+2s}{N}$, then
\begin{align}\label{bgn}
B(u,u,1+\frac{\alpha+2s}{N})\le \frac{1+\frac{\alpha+2s}{N}}{\|Q\|_2^{\frac{2(\alpha+2s)}{N}}}\|(-\Delta)^{\frac{s}{2}}u\|_2^{2}\|u\|_2^{\frac{2(\alpha+2s)}{N}}.
\end{align}
By \eqref{Qequation}, we have
\begin{align}\label{l2bq}
\begin{split}
&\|(-\Delta)^{\frac{s}{2}}Q\|_2^{2}=\frac{N}{N+\alpha+2s}B(Q,Q,1+\frac{\alpha+2s}{N}),\\
&\|Q\|_2^2=\frac{\alpha+2s}{N+\alpha+2s}B(Q,Q,1+\frac{\alpha+2s}{N}).
\end{split}
\end{align}
\end{remark}

\begin{lemma}(Weak Young inequality \cite{LL01}) Let $N\in\mathbb{N},\alpha\in(0,N),p,r>1$ and $\frac{1}{p}=\frac{\alpha}{N}+\frac{1}{r}$. If $v\in L^{p}(\mathbb{R}^N)$, then $I_{\alpha}\star v\in L^{r}(\mathbb{R}^N)$
 and
$$
 \bigg(\int_{\mathbb{R}^N}|I_{\alpha}\star v|^{r}\bigg)^{\frac{1}{r}}\le C(N,\alpha,p)\bigg(\int_{\mathbb{R}^N}|v|^{p}\bigg)^{\frac{1}{p}}.
$$
In particular, we can set $p=\frac{N}{\alpha}$ and $r=+\infty$.
\end{lemma}
\begin{lemma}\label{lIa}
If
$$
\frac{N(2p-1)}{N+\alpha}\le q<\frac{Np}{\alpha},
$$
and $u\in L^q(\mathbb{R}^N)$, then
$$
(I_{\alpha}\star |u|^p)|u|^{p-2}u\in L^r(\mathbb{R}^N)~\text{for~}\frac{1}{r}=\frac{2p-1}{q}-\frac{\alpha}{N}.
$$
\end{lemma}

\begin{lemma}\label{lbw}\cite{Lx20}
Assume that $N>2s,~\alpha\in(0,N)$ and $r\in[1+\frac{\alpha}{N},\frac{N+\alpha}{N-2s}]$. Suppose that $\{u_n\}_{n=1}^{\infty}\subset H^{s}(\mathbb{R}^N)$ satisfy $u_n\rightharpoonup u$  weakly in $H^{s}(\mathbb{R}^N)$ as $n\rightarrow\infty$. Then
\begin{align*}
\int_{\mathbb{R}^N}(I_{\alpha}\star|u_n|^p)|u_n|^{p-2}u_n\varphi\rightarrow\int_{\mathbb{R}^N}(I_{\alpha}\star|u|^p)|u|^{p-2}u\varphi\text{~in }H^{-s}(\mathbb{R}^N)\text{~as~}n\rightarrow\infty,
\end{align*}
for any $\varphi\in H^{s}(\mathbb{R}^N)$.
\end{lemma}
\begin{remark}
Under the assumptions of Lemma \ref{lbw}, if $u_n\rightharpoonup u$ and $v_n\rightharpoonup v$ weakly in $H^{s}(\mathbb{R}^N)$, then
\begin{align*}
&\int_{\mathbb{R}^N}(I_{\alpha}\star|u_n|^p)|v_n|^{p-2}v_n\varphi\rightarrow\int_{\mathbb{R}^N}(I_{\alpha}\star|u|^p)|v|^{p-2}v\varphi\text{~in }H^{-s}(\mathbb{R}^N),\\
&\int_{\mathbb{R}^N}(I_{\alpha}\star|v_n|^p)|u_n|^{p-2}u_n\varphi\rightarrow\int_{\mathbb{R}^N}(I_{\alpha}\star|v|^p)|u|^{p-2}u\varphi\text{~in }H^{-s}(\mathbb{R}^N),
\end{align*}
as $n\rightarrow\infty$, for any $\varphi\in H^{s}(\mathbb{R}^N)$.
\end{remark}

Next we give Br\'ezis-Lieb lemma for the nonlocal term of the functional, which can be seen in \cite{MV13}.
\begin{lemma}\label{lbl}
Let $N\in\mathbb{N},\alpha\in(0,N),p\in[1,\frac{2N}{N+\alpha})$ and $\{u_n\}$ be a bounded sequence in $L^{\frac{2Np}{N+\alpha}}$. If $u_n\to u$ a.e. in $\mathbb{R}^N$ as $n\to\infty$, then
$$
\lim_{n\to\infty}\int_{\mathbb{R}^N}(I_{\alpha}\star|u_n|^{p})|u_n|^{p}-\int_{\mathbb{R}^N}(I_{\alpha}\star|u_n-u|^{p})|u_n-u|^{p}
=\int_{\mathbb{R}^N}(I_{\alpha}\star|u|^{p})|u|^{p}.$$
\end{lemma}
\begin{remark}
Under the assumptions of Lemma \ref{lbl}, if $u_n\to u$ and $v_n\to v$ a.e. in $\mathbb{R}^N$ as $n\to\infty$, then
$$
\lim_{n\to\infty}\int_{\mathbb{R}^N}(I_{\alpha}\star|u_n|^{p})|v_n|^{p}-\int_{\mathbb{R}^N}(I_{\alpha}\star|u_n-u|^{p})|v_n-v|^{p}
=\int_{\mathbb{R}^N}(I_{\alpha}\star|u|^{p})|v|^{p}.$$
\end{remark}

\begin{lemma}\cite{LL01}
Let $f,g$ and $h$ be three Lebesgue measurable non-negative functions on $\mathbb{R}^N$. Then, with
$$
\Psi(f,g,h)=\int_{\mathbb{R}^N}\int_{\mathbb{R}^N}f(x)g(x-y)h(y)dxdy,
$$
we have
$$
\Psi(f,g,h)\le\Psi(f^{*},g^{*},h^{*}),
$$
where $f^{*}$ is the Schwartz rearrangement of $f$.
\end{lemma}

\begin{lemma}\cite{DSS15}\label{lpohozaev}
Let $N\ge 3,s\in(0,1),\alpha\in(0,N),p\in(\max\{2,1+\frac{\alpha}{N}\},\frac{N+\alpha}{N-2s}),\mu\in\mathbb{R},\lambda\in\mathbb{R}$. If $u\in H^{s}(\mathbb{R}^N)$ is a weak solution of
\begin{align*}
(-\Delta)^{s}u+\lambda u=\mu(I_{\alpha}\star|u|^p)|u|^{p-2}u,
\end{align*}
then the Pohozaev identity holds
\begin{align*}
0=Q_{\mu}(u)=\|(-\Delta)^{\frac{s}{2}}u\|_2^2-\mu\delta_{p}\int_{\mathbb{R}^N}(I_{\alpha}\star|u|^p)|u|^p.
\end{align*}
\end{lemma}

In the following, we consider the single equation
\begin{align}\label{scalar problem}
\left\{\begin{aligned}
&(-\Delta)^s u+\lambda u=\mu(I_{\alpha}\star|u|^p)|u|^{p-2}u~\quad x\in \mathbb{R}^N,\\
&\int_{\mathbb{R}^N}u^2dx=c^2.
\end{aligned}
\right.
\end{align}
The corresponding functional with \eqref{scalar problem} is
\begin{align*}
M^{p}_{\mu}(u)=\frac{1}{2}\int_{\mathbb{R}^N}|(-\Delta)^{\frac{s}{2}} u|^2dx-\frac{\mu}{2p}\int_{\mathbb{R}^N}\int_{\mathbb{R}^N}\frac{|u(x)|^{p}|u(y)|^{p}}{|x-y|^{N-\alpha}}dydx.
\end{align*}
Denote that
\begin{align}\label{mcnc}
\begin{split}
&m(c,\mu,p):=\inf_{u\in S_{c}}M^{p}_{\mu}(u)\text{ for }1+\frac{\alpha}{N}<p<1+\frac{\alpha+2s}{N},\\
&n(c,\mu,p):=\inf_{\mathcal{N}_{c}}M^{p}_{\mu}(u) \text{ for }1+\frac{\alpha+2s}{N}<p<\frac{N+\alpha}{N-2s},
\end{split}
\end{align}
where $
\mathcal{N}_{c}:=\{u\in S_{c}:Q_{\mu}(u)=0\}$.

We collect some properties about the scalar equation, which can be seen in \cite{DSS15} and \cite{LL20}.
\begin{lemma}\label{lm}
\begin{itemize}
  \item[(i)]Assume that $1+\frac{\alpha}{N}<p<1+\frac{\alpha+2s}{N}$. Then
  \begin{enumerate}
  \item $m(c,\mu,p)<0$ for all $c>0$.
  \item $m(c,\mu,p)<m(c-\alpha,\mu,p)+m(\alpha,\mu,p)$ for any $\alpha\in(0,c)$.
  \end{enumerate}
  \item[(ii)] Assume that $\max\{2,1+\frac{\alpha+2s}{N}\}<p<\frac{N+\alpha}{N-2s}$. Then
  $n(c,\mu,p)$is strictly decreasing with respect to $c$.
\end{itemize}
\end{lemma}

\begin{lemma}\cite{DSS15}\label{lregularity}
Assume $1+\frac{\alpha}{N}<p<\frac{N+\alpha}{N-2s}$. Then the equation $(-\Delta)^s u+\lambda u=\mu(I_{\alpha}\star|u|^p)|u|^{p-2}u$ has a ground state solution $u\in H^{s}(\mathbb{R}^N)$, which is positive, radially symmetric and decreasing. Furthermore,
\begin{itemize}
\item [(i)] if $s\le\frac12$, then $u\in L^1(\mathbb{R}^N)\cap C^{0,\mu}(\mathbb{R}^N)$ for some $\mu\in(0,2s)$;
 \item [(ii)] if $s>\frac12$, then $u\in L^1(\mathbb{R}^N)\cap C^{1,\mu}(\mathbb{R}^N)$ for some $\mu\in(0,2s-1)$;
 \item [(iii)] if $p\ge 2$, then there exists $C>0$ such that
 $$
 u(x)=\frac{C}{|x|^{N+2s}}+o(|x|^{-(N+2s)}),\text{~as~}|x|\to\infty.
$$
\end{itemize}
\end{lemma}

\begin{lemma}\label{llt}\cite{LHXY20}
Assume that $u\in H^s(\mathbb{R}^N)\cap L^{q}(\mathbb{R}^N)$ with $N>2s,~q\in(0,\frac{N}{N-2s}]$, and $u$ satisfies
\begin{align*}
\left\{\begin{aligned}
&(-\Delta)^{s}u\ge0,\\
&u\ge0.
\end{aligned}
\right.
\end{align*}
Then $u\equiv0$.
\end{lemma}

\section{$L^2$-subcritical case}
In this section, we consider the $L^2$-subcritical case under different potentials.
\subsection{Proof of Theorem \ref{t1}}
In this subsection, if $V_1(x)=V_2(x)=0$, then the corresponding energy functional is
\begin{align*}
E^1_{\mu_1,\mu_2,\beta}(u,v)
=\frac12A(u,v)-\frac{\mu_1}{2p}B(u,u,p)-\frac{\mu_2}{2q}B(v,v,q)-\frac{\beta}{r}B(u,v,r).
\end{align*}
\begin{lemma}\label{lebfb}
$E^1_{\mu_1,\mu_2,\beta}(u,v)$ is bounded from below on $ S_{a}\times S_{b}$.
\end{lemma}
\begin{proof}
By \eqref{gGN} and H\"older inequality, we have
\begin{align*}
E^1_{\mu_1,\mu_2,\beta}(u,v)
\ge&\frac12A(u,v)-\frac{C(N,p,s,\alpha)\mu_1a^{2p(1-\delta_{p})}}{2p}\|(-\Delta)^{\frac{s}{2}}u\|_2^{2p\delta_{p}}\\
&-\frac{C(N,q,s,\alpha)\mu_2b^{2q(1-\delta_{q})}}{2q}\|(-\Delta)^{\frac{s}{2}}v\|_2^{2q\delta_{q}}\\
&-\frac{C(N,r,s,\alpha)\beta a^{r(1-\delta_{r})}b^{r(1-\delta_{r})}}{r}\|(-\Delta)^{\frac{s}{2}}u\|_2^{r\delta_{r}}\|(-\Delta)^{\frac{s}{2}}v\|_2^{r\delta_{r}}\\
\ge&\frac12A(u,v)-\frac{C(N,p,s,\alpha)\mu_1a^{2p(1-\delta_{p})}}{2p}\|(-\Delta)^{\frac{s}{2}}u\|_2^{2p\delta_{p}}\\
&-\frac{C(N,q,s,\alpha)\mu_2b^{2q(1-\delta_{q})}}{2q}\|(-\Delta)^{\frac{s}{2}}v\|_2^{2q\delta_{q}}\\
&-\frac{C(N,r,s,\alpha)\beta a^{2r(1-\delta_{r})}}{2r}\|(-\Delta)^{\frac{s}{2}}u\|_2^{2r\delta_{r}}\\
&-\frac{C(N,r,s,\alpha)\beta b^{2r(1-\delta_{r})}}{2r}\|(-\Delta)^{\frac{s}{2}}v\|_2^{2r\delta_{r}}.
\end{align*}
Since $p,q,r\in(1+\frac{\alpha}{N},1+\frac{\alpha+2s}{N})$, $2p\delta_{p}<2,~2q\delta_{q}<2$ and $2r\delta_{r}<2$. Hence $E^1_{\mu_1,\mu_2,\beta}(u,v)$ is bounded from below. Therefore, we complete the lemma.
\end{proof}
\begin{remark}\label{r1}
Due to $E^1_{\mu_1,\mu_2,\beta}(u,v)$ is bounded from below, we consider the problem
\begin{align*}
e_1(a,b)=\inf_{(u,v)\in S_a\times S_b}E^1_{\mu_1,\mu_2,\beta}(u,v),
\end{align*}
and obtain a minimizing sequence $\{(u_n,v_n)\}\subset S_a\times S_b$ for $e_1(a,b)$. Moreover, $\{(u_n,v_n)\}$ is bounded in $H^s(\mathbb{R}^N)\times H^s(\mathbb{R}^N)$. Hence there exists $(u,v)\in H^s(\mathbb{R}^N)\times H^s(\mathbb{R}^N)$ such that $(u_n,v_n)\rightharpoonup(u,v)$ weakly in $H^s(\mathbb{R}^N)\times H^s(\mathbb{R}^N)$. In order to obtain the strong convergence of minimizing sequence, we consider the work space $H_{r}^{s}(\mathbb{R}^N)\times H_{r}^{s}(\mathbb{R}^N)$. Thus $(u_n,v_n)\to(u,v)$ strongly in $L^{t_1}(\mathbb{R}^N)\times L^{t_2}(\mathbb{R}^N)$ for $t_1,t_2\in(2,2_{s}^{*})$.  By \eqref{gGN}, we can derive that
 \begin{align}\label{Bconvergence}
 \left\{\begin{aligned}
 &B(u_n,u_n,p)\rightarrow B(u,u,p),\\
 &B(v_n,v_n,q)\rightarrow B(v,v,q),\\
 &B(u_n,v_n,r)\rightarrow B(u,v,r).
 \end{aligned}
 \right.
 \end{align}
Furthermore, due to $|(-\Delta)^{\frac{s}{2}}|u||\le |(-\Delta)^{\frac{s}{2}}u|$, we may assume that $u_n\ge0$ and $v_n\ge0$. In a brief conclusion, there exists a nonnegative and radial symmetry minimizing sequence $\{(u_n,v_n)\}\subset S_a\times S_b$ such that $e_1(a,b)=\lim\limits_{n\rightarrow\infty}E^1_{\mu_1,\mu_2,\beta}(u_n,v_n)$.
\end{remark}

From $dE^1_{\mu_1,\mu_2,\beta}|_{S_a\times S_b}(u_n,v_n)\rightarrow0$, there exist two sequences of real numbers $\{\lambda_{1n}\}$ and $\{\lambda_{2n}\}$ such that
\begin{align}\label{dE1}
\begin{split}
&\int_{\mathbb{R}^N}(-\Delta)^{\frac{s}{2}}u_n\varphi dx-\mu_1\int_{\mathbb{R}^N}(I_{\alpha}\star|u_n|^p)|u_n|^{p-2}u_n\varphi dx\\
&-\beta\int_{\mathbb{R}^N}(I_{\alpha}\star|v_n|^r)|u_n|^{r-2}u_n\varphi dx
+\int_{\mathbb{R}^N}(-\Delta)^{\frac{s}{2}}v_n\psi dx\\ &-\mu_2\int_{\mathbb{R}^N}(I_{\alpha}\star|v_n|^q)|v_n|^{q-2}v_n\psi dx-\beta\int_{\mathbb{R}^N}(I_{\alpha}\star|u_n|^r)|v_n|^{r-2}v_n\psi dx\\
&+\lambda_{1n}\int_{\mathbb{R}^N}u_n\varphi dx+\lambda_{2n}\int_{\mathbb{R}^N}v_n\psi dx=o(1)(\|\varphi\|_{H^s}+\|\psi\|_{H^s})
\end{split}
\end{align}
for any $(\varphi,\psi)\in H^s(\mathbb{R}^N)\times H^s(\mathbb{R}^N)$ with $o(1)\rightarrow0$ as $n\rightarrow\infty$.

The following two lemmas will play an important role in the  strong convergence of minimizing sequence in $H^s(\mathbb{R}^N)\times H^s(\mathbb{R}^N)$.
\begin{lemma}\label{llb}
Both $\{\lambda_{1n}\}$ and $\{\lambda_{2n}\}$ are bounded sequences.
\end{lemma}
\begin{proof}
Taking $(u_n,0)$ and $(0,v_n)$ in \eqref{dE1}, we obtain
\begin{align*}
\lambda_{1n}a^2+o(1)
=-\|(-\Delta)^{\frac{s}{2}}u_n\|_2^2+\mu_1B(u_n,u_n,p) +\beta B(u_n,v_n,r),
\end{align*}
and
\begin{align*}
\lambda_{2n}b^2+o(1)
=-\|(-\Delta)^{\frac{s}{2}}v_n\|_2^2+\mu_2B(v_n,v_n,q) +\beta B(u_n,v_n,r),
\end{align*}
where $o(1)\rightarrow0$ as $n\rightarrow\infty$. Due to $\{(u_n,v_n)\}\subset H^s(\mathbb{R}^N)\times H^s(\mathbb{R}^N)$, by \eqref{gGN}, it yields that $B(u_n,u_n,p),~B(v_n,v_n,q)$ and $B(u_n,v_n,r)$ are bounded. Therefore, both $\{\lambda_{1n}\}$ and $\{\lambda_{2n}\}$ are bounded. Furthermore, there exist two real numbers $\lambda_1$ and $\lambda_2$ such that $\lambda_{1n}\rightarrow\lambda_1$ and $\lambda_{2n}\rightarrow\lambda_2$.
\end{proof}

Notice that the sign of $\lambda_{i}(i=1,2)$ plays a key role in the strong convergence of $\{(u_n,v_n)\}$ in $H^s(\mathbb{R}^N)\times H^s(\mathbb{R}^N)$.
\begin{lemma}\label{lsl}
If $\lambda_1>0$ (resp. $\lambda_2>0$), then $u_n\rightarrow u$ (resp. $v_n\rightarrow v$) strongly in $H_{r}^s(\mathbb{R}^N)$.
\end{lemma}
\begin{proof}
Suppose that $\lambda_1>0$. Then applying $u_n\rightharpoonup u$ weakly in $H_{r}^s(\mathbb{R}^N)$, by \eqref{Bconvergence} and \eqref{dE1}, we can obtain that
\begin{align*}
o(1)&=(dE^1_{\mu_1,\mu_2,\beta}(u_n,v_n)-dE^1_{\mu_1,\mu_2,\beta}(u,v))[(u_n-u,0)]+\lambda_1\int_{\mathbb{R}^N}|u_n-u|^2dx\\
&=\int_{\mathbb{R}^N}|(-\Delta)^{\frac{s}{2}}(u_n-u)|^2+\lambda_1\int_{\mathbb{R}^N}|u_n-u|^2dx,
\end{align*}
with $o(1)\rightarrow0$ as $n\rightarrow\infty$. Due to $\lambda_1>0$, the strong  convergence  of $\{u_n\}$ in $H^s(\mathbb{R}^N)$ holds.
\end{proof}

\begin{lemma}\label{le1}
If $e_1(a,b)$ can be arrived, then
\begin{align*}
e_1(a,b)<\min\{m(a,p,\mu_1),m(b,q,\mu_2)\}<0.
\end{align*}
\end{lemma}
\begin{proof}
Suppose that $e_1(a,b)$ can be arrived by $(u,v)$. Then
 $$(u,v)\in S_a\times S_b,E^1_{\mu_1,\mu_2,\beta}(u,v)=e_1(a,b).$$ Since $\beta>0$, we can deduce that
$$e_1(a,b)\le m(a,\mu_1,p)+m(b,\mu_2,q),$$
where $m(a,\mu_1,p)$ and $m(b,\mu_2,q)$ are defined in \eqref{mcnc}. In addition, by Lemma \ref{lm}, we have $m(a,\mu_1,p)<0$ and $m(b,\mu_2,q)<0$. Hence $e_1(a,b)<0$ and the lemma holds.
\end{proof}

Having arrived at the end of this subsection, we give the proof of Theorem \ref{t1}.

\begin{proof}[\textbf{Proof of Theorem \ref{t1}}]
By Remark \ref{r1}, there exist a nonnegative and radial symmetry minimizing sequence $\{(u_n,v_n)\}\subset S_a\times S_b$  and $(u,v)\in H_{r}^s(\mathbb{R}^N)\times H_{r}^s(\mathbb{R}^N)$ with $u\ge0,v\ge0$ such that $e_1(a,b)=\lim\limits_{n\rightarrow\infty}E^1_{\mu_1,\mu_2,\beta}(u_n,v_n)$ and $(u_n,v_n)\rightharpoonup(u,v)$ weakly in $H_{r}^{s}(\mathbb{R}^N)\times H_{r}^{s}(\mathbb{R}^N)$.

In order to complete the proof, we give the following claim.\\
\textbf{Claim 1:} If $(u,v)\in H_{r}^{s}(\mathbb{R}^N)\times H_{r}^{s}(\mathbb{R}^N)$ is a solution of system \eqref{system} with $u\ge0, u\not\equiv 0$ and $v\ge0$, then $\lambda_1>0$. If $(u,v)\in H_{r}^{s}(\mathbb{R}^N)\times H_{r}^{s}(\mathbb{R}^N)$ is a solution of system \eqref{system} with $u\ge0, v\ge0$ and $v\not\equiv0$, then $\lambda_2>0$.

We argue it by contradiction. Indeed, we may assume that $\lambda_1\le0$. From $u\ge0, u\not\equiv0$ and $v\ge0$, it follows that
\begin{align*}
(-\Delta)^s u=-\lambda_1 u+\mu_1(I_{\alpha}\star |u|^p)|u|^{p-2}u+\beta(I_{\alpha}\star |v|^r)|u|^{r-2}u\ge0.
\end{align*}
Combining with Lemma \ref{llt}, we can derive that $u\equiv0$, which contradicts to $u\not\equiv0$. Hence $\lambda_1>0$ and the other case is similar.

We remain to show $(u_n,v_n)\rightarrow(u,v)$ strongly in $H^{s}(\mathbb{R}^N)\times H^{s}(\mathbb{R}^N)$.
Inspired by \textbf{Claim 1} and Lemma \ref{lsl},  we need to show that $u\not\equiv0$ and $v\not\equiv0$. It can be argued by contradiction and there are three cases:
\begin{align*}
\left\{
  \begin{array}{ll}
    \textbf{Case 1}: & \hbox{$u\equiv0,v\equiv0$;} \\
    \textbf{Case 2}: & \hbox{$u\not\equiv0,v\equiv0$;} \\
   \textbf{Case 3}: & \hbox{$u\equiv0,v\not\equiv0$.}
  \end{array}
\right.
\end{align*}

For \textbf{Case 1},  by \eqref{gGN}, we have $B(u_n,u_n,p)\rightarrow0,~B(v_n,v_n,q)\rightarrow0$ and $B(u_n,v_n,r)\rightarrow0$. Thus $\lim\limits_{n\rightarrow\infty}E^1_{\mu_1,\mu_2,\beta}(u_n,v_n)=0$, it contradicts to $e_1(a,b)<0$. Hence \textbf{Case 1} does not occur. For \textbf{Case 2}, by \eqref{gGN}, it yields that $B(v_n,v_n,q)\rightarrow0$ and $B(u_n,v_n,r)\rightarrow0$. Therefore,
\begin{align*}
e_1(a,b)=E^1_{\mu_1,\mu_2,\beta}(u,v)=\frac12A(u,v)-\frac{\mu_1}{2p}B(u,u,p)\ge m(a,\mu_1,p),
\end{align*}
which is a contradiction with Lemma \ref{le1}. Therefore, \textbf{Case 2} doesn't hold. Similarly, \textbf{Case 3} doesn't hold. In conclusion, $u\not\equiv0$ and $v\not\equiv0$.

From \textbf{Claim 1}, we can conclude that $\lambda_1>0$ and
$\lambda_2>0$. Therefore, by Lemma \ref{lsl},  $(u_n,v_n)\rightarrow(u,v)$ strongly in $H^{s}(\mathbb{R}^N)\times H^{s}(\mathbb{R}^N)$. Hence we finish the proof.
\end{proof}

\subsection{Proof of Theorem \ref{t11}}
Under the assumptions \textbf{(V2)}, we define
\begin{align}\label{Hss}
H^s_{i}(\mathbb{R}^N)=\{u\in H^s(\mathbb{R}^N):\int_{\mathbb{R}^N}V_{i}(x)|u|^2dx<\infty\}(i=1,2).
\end{align}
Then $H^s_{i}(\mathbb{R}^N)(i=1,2)$ is a  Hilbert space equipped with the norms
$$
\|u\|_{i}=\bigg(\int_{\mathbb{R}^N}(|(-\Delta)^{\frac{s}{2}}u|^2+V_{i}(x)|u|^2)dx\bigg)^{\frac12}(i=1,2).
$$
In addition, the following key compact embedding lemma can be seen in \cite{DTW19}.
\begin{lemma}\label{lce}
The embedding $H^s_1(\mathbb{R}^N)\times H^s_2(\mathbb{R}^N)\hookrightarrow L^{p}(\mathbb{R}^N)\times L^{p}(\mathbb{R}^N)$ is compact for $p\in[2,\frac{2N}{N-2s})$.
\end{lemma}
\begin{remark}\label{rbunvn}
If there exists a sequence $\{(u_n,v_n)\}\subset H^s_1(\mathbb{R}^N)\times H^s_2(\mathbb{R}^N)$ such that $(u_n,v_n)\rightharpoonup(u,v)$ weakly in $H^s_1(\mathbb{R}^N)\times H^s_2(\mathbb{R}^N)$ for some $(u,v)\in H^s_1(\mathbb{R}^N)\times H^s_2(\mathbb{R}^N)$. By Lemma \ref{lhls} and Lemma \ref{lce} for $p\in(1+\frac{\alpha}{N},\frac{N+\alpha}{N-2s})$, there holds that
\begin{align*}
B(u_n-u,u_n-u,p)\le C(N,\alpha)\|u_n-u\|^{2p}_{\frac{2Np}{N+\alpha}}\rightarrow0.
\end{align*}
Similarly,
\begin{align*}
B(v_n-v,v_n-v,q)\le C(N,\alpha)\|v_n-v\|^{2q}_{\frac{2Nq}{N+\alpha}}\rightarrow0,
\end{align*}
and
\begin{align*}
B(u_n-u,v_n-v,q)\le C(N,\alpha)\|u_n-u\|^{r}_{\frac{2Nr}{N+\alpha}}\|v_n-v\|^{r}_{\frac{2Nr}{N+\alpha}}\rightarrow0,
\end{align*}

\end{remark}
In order to find the normalized solutions for system \eqref{system}, we shall consider the minimizer of the following constrained problem
\begin{align*}
e_2(a,b)=\inf_{(u,v)\in F}E_2(u,v),
\end{align*}
where
$$
E_2(u,v)=E_{\mu_1,\mu_2,\beta}(u,v):H^s_1(\mathbb{R}^N)\times H^s_2(\mathbb{R}^N)\rightarrow\mathbb{R}
$$
and
$F=[H^s_1(\mathbb{R}^N)\times H^s_2(\mathbb{R}^N)]\cap[S_a\times S_b]$.
\begin{proof}[\textbf{Proof of Theorem \ref{t11}}] Similar to Lemma \ref{lebfb}, we can derive that
\begin{align*}
E_2(u,v)\ge&\frac12A(u,v)+\frac12\int_{\mathbb{R}^N}(V_1(x)|u|^2+V_2(x)|v|^2)-\frac{C(a)\mu_1}{2p}\|(-\Delta)^{\frac{s}{2}}u\|_2^{2p\delta_{p}}\\
&-\frac{C(b)\mu_2}{2q}\|(-\Delta)^{\frac{s}{2}}v\|_2^{2q\delta_{q}}-\frac{C(a,b)\beta}{r}\|(-\Delta)^{\frac{s}{2}}u\|_2^{r\delta_{r}}\|(-\Delta)^{\frac{s}{2}}v\|_2^{r\delta_{r}}\\
\ge&\frac12A(u,v)+\frac12\int_{\mathbb{R}^N}(V_1(x)|u|^2+V_2(x)|v|^2)-\frac{C(a)\mu_1}{2p}\|(-\Delta)^{\frac{s}{2}}u\|_2^{2p\delta_{p}}\\
&-\frac{C(b)\mu_2}{2q}\|(-\Delta)^{\frac{s}{2}}v\|_2^{2q\delta_{q}}-\frac{C(a,b)\beta}{2r}\|(-\Delta)^{\frac{s}{2}}u\|_2^{2r\delta_{r}}-\frac{C(a,b)\beta}{2r}\|(-\Delta)^{\frac{s}{2}}v\|_2^{2r\delta_{r}}.
\end{align*}
Since $p,q,r\in(1+\frac{\alpha}{N},1+\frac{\alpha+2s}{N})$, $2p\delta_{p}<2,2q\delta_{q}<2,2r\delta_{r}<2$. Therefore, $E_2(u,v)$ is bounded from below and coercive on the set $F$.

Due to $V_{i}\ge0(i=1,2)$,  similar to Theorem \ref{t1}, there exists a nonnegative minimizing sequence $\{(u_n,v_n)\}\subset F$. Furthermore, there exists a $(u,v)\in H^s_1(\mathbb{R}^N)\times H^s_2(\mathbb{R}^N)$ with $u,v\ge0$ such that $(u_n,v_n)\rightharpoonup(u,v)$ weakly in $H^s_1(\mathbb{R}^N)\times H^s_2(\mathbb{R}^N)$. From Lemma \ref{lce}, it follows that $(u,v)\in S_a\times S_b$, i.e., $(u,v)\in F$. Moreover $E_2(u,v)\ge e_2(a,b)$. By Lemma \ref{lbw} and Remark \ref{rbunvn}, we can obtain that
\begin{align*}
B(u_n,u_n,p)\rightarrow B(u,u,p),~B(v_n,v_n,q)\rightarrow B(v,v,q)\text{~and~}B(u_n,v_n,r)\rightarrow B(u,v,r).
\end{align*}
By Fatou's Lemma, it can be concluded that $e_2(a,b)\ge E_2(u,v)$. Therefore, $e_2(a,b)=E_2(u,v)$. Moreover, $u\ge0$ and $v\ge0$.
\end{proof}
In the following, we consider  $V_{i}(x)(i=1,2)$ satisfies
\begin{align}\label{v3}
1\le V_{i}(x)\in L_{loc}(\mathbb{R}^N)\text{~and~}\lim_{|x|\to\infty}V_{i}(x)=\infty,~~i=1,2.
\end{align}
\begin{lemma}\label{lu0v0}
Suppose that $\max\{2,1+\frac{\alpha}{N}\}<p,q,r<1+\frac{\alpha+2s}{N}$ and  $V_{i}(x)(i=1,2)$ satisfies \eqref{v3}. Let $(u_0,v_0)\in H^{s}(\mathbb{R}^N)\times H^{s}(\mathbb{R}^N)$ be the nonnegative solution for the system \eqref{system}. Then $u_0,v_0\in L^{\infty}(\mathbb{R}^N)$ and $u_0(x),v_0(x)\to 0$ as $|x|\to\infty$.
\end{lemma}
\begin{proof}
Let $(\hat{u},\hat{v})\in H^{s}(\mathbb{R}^N)\times H^{s}(\mathbb{R}^N)$ be the solution of the following coupled system
\begin{align}\label{fs1}
\left\{\begin{aligned}
&(-\Delta)^s u+u=-\lambda_1u_0+\mu_1(I_{\alpha}\star |u_0|^p)|u_0|^{p-2}u_0+\beta(I_{\alpha}\star |v_0|^r)|u_0|^{r-2}u_0:=f_1(x),\\
&(-\Delta)^s v+v=-\lambda_2v_0+\mu_2(I_{\alpha}\star |v_0|^q)|v_0|^{q-2}v_0+\beta(I_{\alpha}\star |u_0|^r)|v_0|^{r-2}v_0:=f_2(x).
\end{aligned}
\right.
\end{align}
By Lemma \ref{lhls}, we have $\hat{u}\in L^{\chi(p)}(\mathbb{R}^N)\cap L^{\chi(r)}(\mathbb{R}^N)$ and $\hat{v}\in L^{\chi(q)}(\mathbb{R}^N)\cap L^{\chi(r)}(\mathbb{R}^N)$ for $\chi(p)=\frac{2Np}{N+\alpha}$. Moreover,
$$
\frac{N(2p-1)}{N+\alpha}\le \chi(p)<\frac{Np}{\alpha},~\frac{N(2r-1)}{N+\alpha}\le \chi(r)<\frac{Nr}{\alpha}\text{~and~}\frac{N(2q-1)}{N+\alpha}\le \chi(q)<\frac{Nq}{\alpha}.
$$
From Lemma \ref{lIa}, we can derive that
\begin{align*}
\left\{\begin{aligned}
&(I_{\alpha}\star |u_0|^p)|u_0|^{p-2}u_0\in L^{t(p)}(\mathbb{R}^N),\\
&(I_{\alpha}\star |v_0|^r)|u_0|^{r-2}u_0\in L^{t(r)}(\mathbb{R}^N),\\
&(I_{\alpha}\star |v_0|^q)|v_0|^{q-2}v_0\in L^{t(q)}(\mathbb{R}^N),
\end{aligned}
\right.
\end{align*}
where $\frac{1}{t(p)}=\frac{2p-1}{\chi(p)}-\frac{\alpha}{N}$.
Direct computation shows that
\begin{align*}
&\min\{t(p),t(q),t(r),2\}>\frac{N}{2s}.
\end{align*}
Denote that
$$
\Theta_0=\max\{t(p),t(q),t(r),2\}.
$$
Let $\mathcal{K}$ be the Bessel kernel
$$
\mathcal{K}(x)=\mathcal{F}^{-1}\bigg(\frac{1}{1+|\xi|^{2s}}\bigg),
$$
where $\mathcal{F}^{-1}g(\xi)$ stand for the inverse  Fourier transform of $g(\xi)$. Then
$$
\hat{u}(x)=(\mathcal{K}\star f_1)(x)=\int_{\mathbb{R}^N}\mathcal{K}(x-y)f_1(y)dy,~\hat{v}(x)=(\mathcal{K}\star f_2)(x)=\int_{\mathbb{R}^N}\mathcal{K}(x-y)f_2(y)dy.
$$
From Lemma 5.2 in \cite{FQT12}, it yields that $\hat{u},\hat{v}\in L^{\Theta}(\mathbb{R}^N)\cap L^{\infty}(\mathbb{R}^N)$ for any $\Theta>\Theta_0$. Hence the solutions for the coupled system \eqref{fs1} is well-defined. Therefore,
\begin{align*}
\left\{\begin{aligned}
&(-\Delta)^{s}(\hat{u}-u_0)+(\hat{u}-u_0)=(V_1(x)-1)u_0\ge0,\\
&(-\Delta)^{s}(\hat{v}-v_0)+(\hat{v}-v_0)=(V_2(x)-1)v_0\ge0.
\end{aligned}
\right.
\end{align*}
It is easy to observe that $0\le u_0\le \hat{u}$ and $0\le v_0\le \hat{v}$.

By Lemma \ref{lregularity}, there exists a constant $\eta\in(0,1)$ such that $\hat{u},\hat{v}\in C^{0,\eta}(\mathbb{R}^N)$. Furthermore, the fact that $\hat{u},\hat{v}\in L^{\Theta}(\mathbb{R}^N)\cap C^{0,\eta}(\mathbb{R}^N)$ implies that
$\hat{u}(x),\hat{v}(x)\to0$ as $|x|\to\infty$.  Therefore,  $u_0,v_0\in L^{\infty}(\mathbb{R}^N)$ and $u_0(x),v_0(x)\to0$ as $|x|\to\infty$.

\end{proof}
\begin{lemma}\label{ldecay}
Under the assumptions of Lemma \ref{lu0v0}, let $(u_0,v_0)\in H^{s}(\mathbb{R}^N)\times H^{s}(\mathbb{R}^N)$ be the nonnegative solution for the system \eqref{system}. Then there exists a constant $C>0$ such that
$$
u_0,v_0\le \frac{C}{|x|^{N+2s}},~|x|\ge1.
$$
\end{lemma}
\begin{proof}
By the above lemma, we know that $u_0, v_0$ are bounded and $u_0(x)\to0,v_0(x)\to0$ as $|x|\to\infty$. We rewrite \eqref{system} as
\begin{align*}
\left\{\begin{aligned}
&(-\Delta)^s u_0+\frac12u_0=(\frac12-\lambda_1+V_1(x)+\beta(I_{\alpha}\star |v_0|^r)|u_0|^{r-2})u_0+\mu_1(I_{\alpha}\star |u_0|^p)|u_0|^{p-2}u_0,\\
&(-\Delta)^s v_0+\frac12v_0=(\frac12-\lambda_2+V_2(x)++\beta(I_{\alpha}\star |u_0|^r)|v_0|^{r-2})v_0+\mu_2(I_{\alpha}\star |v_0|^q)|v_0|^{q-2}u_0.
\end{aligned}
\right.
\end{align*}
From the facts $V_{i}(x)\to0(i=1,2),u_0(x)\to0,v_0(x)\to0$ as $|x|\to\infty$ and $u_0\ge0,v_0\ge0$, we can  conclude that there exists a constant $R_0>0$ large enough such that
\begin{align}\label{uvle0}
\left\{\begin{aligned}&(-\Delta)^s u_0\le0,\\
&(-\Delta)^s v_0\le0,
\end{aligned}
\right. ~~|x|\ge R_0,
\end{align}
and
\begin{align*}
\left\{\begin{aligned}&(-\Delta)^s u_0+\frac12u_0\ge0,\\
&(-\Delta)^s v_0+\frac12v_0\ge0,
\end{aligned}
\right. ~~|x|\ge R_0+1.
\end{align*}

By Lemma 4.3 in \cite{FQT12}, there is a continuous functions $w$ in $R^N$ satisfying
$$
(-\Delta)^s w(x)+\frac{1}{2} w(x)=0 \quad \text {~if~}|x|>1
$$
in the classical sense, and
$$
0<w(x) \le \frac{c_1}{|x|^{N+2 \alpha}}
$$
for an appropriate $c_1>0$.

Since $u_0,v_0\in L^{\infty}(\mathbb{R}^N)$ and $w(x)>0$, we can choose a constant $C>0$ large enough such that
\begin{align*}
\left\{\begin{aligned}
&w_1(x)=Cw(x)-u_0>0\\
&w_2(x)=Cw(x)-v_0>0
\end{aligned}
\right.\quad\text{if~}|x|\le R_0+1.
\end{align*}
From \eqref{uvle0} and Proposition 2.15 in \cite{S07}, it can be obtained that $u_0,v_0$ are upper-continuous functions for $|x|>R_0$. Therefore, $w_1,w_2$ are lower-continuous functions for $|x|\ge R_0+1$ and
$$
(-\Delta)^s w_{i}(x)+\frac{1}{2} w_{i}(x)\ge0 \quad \text {~if~}|x|>R_0+1.
$$

In the following, we claim that $w_{i}(x)\ge0(i=1,2)$ for any $x\in\mathbb{R}^N$. It can be argued by contradiction and we may assume that $w_1(x)$ is negative in some points, denote that $\Omega_1=\{x\in\mathbb{R}^N:w_1(x)<0\}$. It is easy to find that $\bar\Omega_1\subset\{x\in\mathbb{R}^N:|x|>R_0+1\}$ and
\begin{align*}
\left\{\begin{aligned}
&(-\Delta)^s w_1(x)\ge0, &x\in\Omega_1,\\
&w_1(x)\ge0,&x\in\Omega^{c}_1.
\end{aligned}
\right.
\end{align*}
Since $w_1(x)$ is lower-continuous function in $\bar\Omega_1$, by Proposition 2.17 in \cite{S07}, we conclude that $w_1\ge0$ for any $x\in\mathbb{R}^N$, which implies a contradiction. By the same way, we can obtain the claim for $w_2(x)$. Therefore,
$$
u_0(x),v_0(x)\le Cw(x)\le\frac{Cc_1}{|x|^{N+2s}},~~|x|\ge1.
$$
Hence the lemma holds.
\end{proof}
\begin{proof}[\textbf{Proof of Theorem \ref{t12}}] Let $(u_0,v_0)$ be the nonnegative solution obtained in Theorem \ref{t11}.
By Lemma \ref{ldecay}, we have
$$
u_0,v_0\le \frac{C}{|x|^{N+2s}},~|x|\ge1,
$$
In addition, $(u_0,v_0)$ satisfies
\begin{align*}
\left\{\begin{aligned}
&(-\Delta)^s u_0+u_0=f_1(x),\\
&(-\Delta)^s v_0+v_0=f_2(x),
\end{aligned}
\right.
\end{align*}
where
$$
f_1(x)=(1-V_1(x)-\lambda_1)u_0+\mu_1(I_{\alpha}\star |u_0|^p)|u_0|^{p-2}u_0+\beta(I_{\alpha}\star |v_0|^r)|u_0|^{r-2}u_0,
$$
and
$$
f_2(x)=
(1-V_2(x)-\lambda_2)v_0+\mu_2(I_{\alpha}\star |v_0|^q)|v_0|^{q-2}v_0+\beta(I_{\alpha}\star |u_0|^r)|v_0|^{r-2}v_0.
$$

By the decay estimate of $u_0,v_0$ and the assumptions of $V_{i}(x)(i=1,2)$, there holds that $f_{i}(x)\in L^{p}(\mathbb{R}^N)\cap L^{\infty}(\mathbb{R}^N)(i=1,2)$ for any $p>\frac{N}{N+2s-\theta}$. By Theorem 3.4 in \cite{FQT12}, there exists a constant $\alpha\in(0,1)$ such that $u_0,v_0\in C^{0,\alpha}(\mathbb{R}^N)$. Combining with $u_0,v_0\ge0$, we deduce that $u_0$ and $v_0$ are positive. Consequently, the theorem is proved.

\end{proof}

\section{$L^2$-critical}
In section, we consider the existence of normalized solutions for $L^2$-critical case $p=q=r=1+\frac{\alpha+2s}{N}$. For seeking normalized solutions, we consider the following minimization problem on the constraint
\begin{align}\label{p3}
e_3(a,b)=\inf_{(u,v)\in F}E_{\mu_1,\mu_2,\beta}(u,v),~ F=[H^s_1(\mathbb{R}^N)\times H^s_2(\mathbb{R}^N)]\cap[S_a\times S_b]
\end{align}
We start with introducing the following auxiliary minimization problem
\begin{align}\label{o}
\mathcal{O}(\mu_1,\mu_2,\beta)=\inf_{(u,v)\in S_a\times S_b}\frac{A(u,v)}{\frac{\mu_1}{p}B(u,u,p)+\frac{\mu_2}{p}B(v,v,p)+\frac{2\beta}{p}B(u,v,p)}.
\end{align}
By analyzing \eqref{o}, the criteria about the existence of minimizer for \eqref{p3} can be given.
\begin{proposition}\label{propo1}
Let $\mu_1,\mu_2,\beta>0$. Then
\begin{itemize}
  \item [(i)] \eqref{p3} has at least one minimizer if $\mathcal{O}(\mu_1,\mu_2,\beta)>1$,
  \item [(ii)] \eqref{p3} has no minimizer if $\mathcal{O}(\mu_1,\mu_2,\beta)<1$.
\end{itemize}
\end{proposition}
Before the proof of Proposition \ref{propo1}, we need the following lemma.
\begin{lemma}
Let $\mathcal{O}(\mu_1,\mu_2,\beta)$ be defined by \eqref{o}. Then $\mathcal{O}(\mu_1,\mu_2,\beta)$ is locally Lipschitz continuous in $\mathbb{R}^{3}_{+}$.
\end{lemma}
\begin{proof}
First we prove that
\begin{align}\label{bound}
\frac{a^{*}}{\max\{(\mu_1+\beta)a^{2(p-1)},(\mu_2+\beta)b^{2(p-1)}\}}\le \mathcal{O}(\mu_1,\mu_2,\beta)\le\frac{(a^2+b^2)a^{*}}{\mu_1a^{2p}+\mu_2b^{2p}+2\beta a^pb^p}.
\end{align}
Indeed, from \eqref{gGN} and \eqref{bgn}, it can be derived that
\begin{align}\label{buvgn}
\begin{split}
B(u,v,p)
&\le C(N,\alpha)\|u\|^{p}_{\frac{2Np}{N+\alpha}}\|v\|^{p}_{\frac{2Np}{N+\alpha}}\le\frac{C(N,\alpha)}{2}\bigg[\|u\|^{2p}_{\frac{2Np}{N+\alpha}}+\|v\|^{2p}_{\frac{2Np}{N+\alpha}}\bigg]\\
&\le\frac{p}{2a^{*}}\bigg[\|(-\Delta)^{\frac{s}{2}}u\|^2_2a^{2(p-1)}+\|(-\Delta)^{\frac{s}{2}}v\|^2_2b^{2(p-1)}\bigg].
\end{split}
\end{align}
By Lemma \ref{lggn} and \eqref{buvgn}, we hace
\begin{align*}
&\frac{A(u,v)}{\frac{\mu_1}{p}B(u,u,p)+\frac{\mu_2}{p}B(v,v,p)+\frac{2\beta}{p}B(u,v,p)}\\
\ge&\frac{A(u,v)a^{*}}{(\mu_1+\beta)a^{2(p-1)}\|(-\Delta)^{\frac{s}{2}}u\|_2^2+(\mu_2+\beta)b^{2(p-1)}\|(-\Delta)^{\frac{s}{2}}v\|_2^2}\\
\ge&\frac{a^{*}}{\max\{(\mu_1+\beta)a^{2(p-1)},(\mu_2+\beta)b^{2(p-1)}\}}.
\end{align*}
For the upper bound of $\mathcal{O}(\mu_1,\mu_2,\beta)$,  we take $(a\frac{Q}{\|Q\|_2},b\frac{Q}{\|Q\|_2})\in S_a\times S_b$ as a test function for \eqref{o}, it can be derived from \eqref{l2bq} that
\begin{align*}
&\frac{A(a\frac{Q}{\|Q\|_2},b\frac{Q}{\|Q\|_2})}{\frac{\mu_1}{p}B(a\frac{Q}{\|Q\|_2},a\frac{Q}{\|Q\|_2},p)+\frac{\mu_2}{p}B(b\frac{Q}{\|Q\|_2},b\frac{Q}{\|Q\|_2},p)+\frac{2\beta}{p}B(a\frac{Q}{\|Q\|_2},b\frac{Q}{\|Q\|_2},p)}\\
=&\frac{(a^2+b^2)pa^{*}}{\mu_1a^{2p}+\mu_2b^{2p}+2\beta a^pb^p}\frac{\|(-\Delta)^{\frac{s}{2}}Q\|_2^2}{B(Q,Q,p)}\\
=&\frac{(a^2+b^2)a^{*}}{\mu_1a^{2p}+\mu_2b^{2p}+2\beta a^pb^p}.
\end{align*}
Hence we can give the upper bound of $\mathcal{O}(\mu_1,\mu_2,\beta)$.

We consider $(\mu_1,\mu_2,\beta),(\hat{\mu}_1,\hat{\mu}_2,\hat{\beta})\in\mathbb{R}^{3}_{+}$. Let $\{(u_n,v_n)\}\subset S_a\times S_b$ be a minimizing sequence of $\mathcal{O}(\mu_1,\mu_2,\beta)$. It is easy to observe that $\mathcal{O}(\mu_1,\mu_2,\beta)$ is invariant under the scaling $(u_{t}(x),v_{t}(x))=(t^{\frac{N}{2}}u(tx),t^{\frac{N}{2}}v(tx))$ for any $t>0$. Without loss of generality, we may assume that
$$
A(u_n,v_n)=1,\text{~for all }n.
$$
By Lemma \ref{lggn} and \eqref{buvgn}, it yields that
\begin{align*}
&\frac{1}{\mathcal{O}(\mu_1,\mu_2,\beta)}\\
=&\lim_{n\rightarrow\infty}\frac{\frac{\hat{\mu}_1}{p}B(u_n,u_n,p)+\frac{\hat{\mu}_2}{p}B(v_n,v_n,p)+\frac{2\hat{\beta}}{p}B(u_n,v_n,p)}{A(u_n,v_n)}\\
&+\frac{\mu_1-\hat{\mu}_1}{p}B(u_n,u_n,p)+\frac{\mu_2-\hat{\mu}_2}{p}B(v_n,v_n,p)+\frac{2(\beta-\hat{\beta})}{p}B(u_n,v_n,p)\\
\le&\frac{1}{\mathcal{O}(\hat{\mu}_1,\hat{\mu}_2,\hat{\beta})}+\frac{a^{2(p-1)}}{a^{*}}|\mu_1-\hat{\mu}_1|+\frac{b^{2(p-1)}}{a^{*}}|\mu_2-\hat{\mu}_2|+\frac{a^{2(p-1)}+b^{2(p-1)}}{a^{*}}|\beta-\hat{\beta}|.
\end{align*}
As a consequence,
\begin{align*}
\frac{1}{\mathcal{O}(\mu_1,\mu_2,\beta)}-\frac{1}{\mathcal{O}(\hat{\mu}_1,\hat{\mu}_2,\hat{\beta})}\le\frac{3(a^{2(p-1)}+b^{2(p-1)})}{a^{*}}|(\mu_1,\mu_2,\beta)-(\hat{\mu}_1,\hat{\mu}_2,\hat{\beta})|.
\end{align*}
Therefore,
\begin{align*}
\bigg|\frac{\mathcal{O}(\mu_1,\mu_2,\beta)-\mathcal{O}(\hat{\mu}_1,\hat{\mu}_2,\hat{\beta})}{\mathcal{O}(\mu_1,\mu_2,\beta)\mathcal{O}(\hat{\mu}_1,\hat{\mu}_2,\hat{\beta})}\bigg|&=\bigg|\frac{1}{\mathcal{O}(\mu_1,\mu_2,\beta)}-\frac{1}{\mathcal{O}(\hat{\mu}_1,\hat{\mu}_2,\hat{\beta})}\bigg|\\
&\le\frac{3(a^{2(p-1)}+b^{2(p-1)})}{a^{*}}|(\mu_1,\mu_2,\beta)-(\hat{\mu}_1,\hat{\mu}_2,\hat{\beta})|.
\end{align*}
By \eqref{bound}, we can conclude that
\begin{align*}
&|\mathcal{O}(\mu_1,\mu_2,\beta)-\mathcal{O}(\hat{\mu}_1,\hat{\mu}_2,\hat{\beta})|\\
\le&\frac{3(a^{2(p-1)}+b^{2(p-1)})(a^2+b^2)^2a^{*}}{(\mu_1a^{2p}+\mu_2b^{2p}+2\beta a^pb^p)(\hat{\mu}_1a^{2p}+\hat{\mu}_2b^{2p}+2\hat{\beta}a^pb^p)}|(\mu_1,\mu_2,\beta)-(\hat{\mu}_1,\hat{\mu}_2,\hat{\beta})|.
\end{align*}
Therefore, we establish the lemma.
\end{proof}

\begin{proof}[\textbf{Proof of Proposition \ref{propo1}}]
(i) Let $\{(u_n,v_n)\}\subset H^s_1(\mathbb{R}^N)\times H^s_2(\mathbb{R}^N)$ be a minimizing sequence for \eqref{p3}, where $H^s_{i}(\mathbb{R}^N)(i=1,2)$ are defined by \eqref{Hss}. Then
$$
\|u_n\|_2^2=a^2,~\|v_n\|_2^2=b^2{~and~}\lim_{n\rightarrow\infty}E_{\mu_1,\mu_2,\beta}(u_n,v_n)=e_3(a,b).
$$
By \eqref{efct}, there holds that
$$
E_{\mu_1,\mu_2,\beta}(u_n,v_n)\ge\frac12(1-\frac{1}{\mathcal{O}(\mu_1,\mu_2,\beta)})A(u,v)+\frac12\int_{\mathbb{R}^N}(V_1(x)|u_n|^2+V_2(x)|v_n|^2)dx.
$$
If $\mathcal{O}(\mu_1,\mu_2,\beta)>1$, then $\{(u_n,v_n)\}\subset H^s_1(\mathbb{R}^N)\times H^s_2(\mathbb{R}^N)$ is bounded. Hence there exists a $(u,v)\in H^s_1(\mathbb{R}^N)\times H^s_2(\mathbb{R}^N)$ such that
\begin{align*}
 (u_n,v_n)\rightharpoonup(u,v)\text{~weakly in }H^s_1(\mathbb{R}^N)\times H^s_2(\mathbb{R}^N).
\end{align*}
In addition, by  Remark \ref{rbunvn} and Lemma \ref{lce} ,  it can be concluded that $(u,v)\in S_a\times S_b$ and
\begin{align*}
B(u_n,u_n,p)\rightarrow B(u,u,p),B(v_n,v_n,p)\rightarrow B(v,v,p)\text{~and~}B(u_n,v_n,p)\rightarrow B(u,v,p).
\end{align*}
Similar to the proof of Theorem \ref{t11}, we can obtain that $e_3(a,b)=E_{\mu_1,\mu_2,\beta}(u,v)$. In other words, \eqref{p3} has at least one minimizer. Hence (i) holds.

(ii) Suppose that $\mathcal{O}(\mu_1,\mu_2,\beta)>1$. From the property of infimum, there exists  $(u,v)\in F$ and $u,v$ have compact support in $\mathbb{R}^N$ such that
\begin{align}\label{o1}
\frac{A(u,v)}{\frac{\mu_1}{p}B(u,u,p)+\frac{\mu_2}{p}B(v,v,p)+\frac{2\beta}{p}B(u,v,p)}\le\delta=\frac{1+\mathcal{O}(\mu_1,\mu_2,\beta)}{2}<1,
\end{align}
where $F=[H^s_1(\mathbb{R}^N)\times H^s_2(\mathbb{R}^N)]\cap[S_a\times S_b]$.
For $t>0$, define that $(u_{t}(x),v_{t}(x))=(t^{\frac{N}{2}}u(tx),t^{\frac{N}{2}}v(tx))$. It is obvious to check that $(u_{t}(x),v_{t}(x))\in F$. From $u,v$ have compact support and $V_{i}(x)\in L^{\infty}_{loc}(\mathbb{R}^N)$, it follows that for some constants $C>0$,
\begin{align}\label{vc}
\int_{\mathbb{R}^N}[V_1(x)|u_t(x)|^2+V_2(x)|v_t(x)|^2]dx=\int_{\mathbb{R}^N}[V_1(\frac{x}{t})|u|^2+V_2(\frac{x}{t})|v|^2]dx\le C.
\end{align}
From \eqref{o1} and \eqref{vc}, direct computation shows that
\begin{align*}
E_{\mu_1,\mu_2,\beta}(u_{t}(x),v_{t}(x))
=&t^{2s}\bigg(\frac{1}{2}A(u,v)-\frac{\mu_1}{2p}B(u,u,p)-\frac{\mu_2}{2p}B(v,v,p)-\frac{\beta}{p}B(u,v,p)\bigg)\\
&+\frac12\int_{\mathbb{R}^N}[V_1(\frac{x}{t})|u|^2+V_2(\frac{x}{t})|v|^2]dx\\
\le&t^{2s}\bigg(\frac{1}{2}A(u,v)-\frac{\mu_1}{2p}B(u,u,p)-\frac{\mu_2}{2p}B(v,v,p)-\frac{\beta}{p}B(u,v,p)\bigg)+C\\
\le&t^{2s}(1-\frac{1}{\delta})A(u,v)+C\\
\rightarrow&-\infty,\quad\text{as~}t\rightarrow\infty.
\end{align*}
Therefore,
\begin{align*}
e_3(a,b)\le E_{\mu_1,\mu_2,\beta}(u_{t}(x),v_{t}(x))\rightarrow-\infty,\quad\text{as~}t\rightarrow\infty,
\end{align*}
which implies \eqref{p3} has no minimizer.  In conclusion, the proposition holds.
\end{proof}
We end the section with the proof of Theorem \ref{t2} and Theorem \ref{t21}.
\begin{proof}[\textbf{Proof of Theorem \ref{t2}}]
$(i)$ For any $(u,v)\in S_{a}\times S_{b}$, from \eqref{gGN} and \eqref{bgn}, we can derive that
\begin{align*}
B(u,v,p)&\le C(N,\alpha)\|u\|^{p}_{\frac{2Np}{N+\alpha}}\|v\|^{p}_{\frac{2Np}{N+\alpha}}\le\frac{p}{a^{*}}\|(-\Delta)^{\frac{s}{2}}u\|_2\|(-\Delta)^{\frac{s}{2}}v\|_2a^{p-1}b^{p-1}.
\end{align*}
Consequently,
\begin{align*}
\mathcal{O}(\mu_1,\mu_2,\beta)
&=\inf_{(u,v)\in S_a\times S_b}\frac{A(u,v)}{\frac{\mu_1}{p}B(u,u,p)+\frac{\mu_2}{p}B(v,v,p)+\frac{2\beta}{p}B(u,v,p)}\\
&\ge\inf_{(u,v)\in S_a\times S_b}\frac{a^{*}(\|(-\Delta)^{\frac{s}{2}}u\|_2^2+\|(-\Delta)^{\frac{s}{2}}v\|_2^2)}{W(u,v)},
\end{align*}
where
$$
W(u,v)=\mu_1a^{2(p-1)}\|(-\Delta)^{\frac{s}{2}}u\|_2^2+\mu_2b^{2(p-1)}\|(-\Delta)^{\frac{s}{2}}v\|_2^2+2\beta\|(-\Delta)^{\frac{s}{2}}u\|_2\|(-\Delta)^{\frac{s}{2}}v\|_2a^{p-1}b^{p-1}.
$$
Let
$$
t=\frac{\|(-\Delta)^{\frac{s}{2}}v\|_2}{\|(-\Delta)^{\frac{s}{2}}u\|_2}\in(0,\infty)\text{~and~}f_{\mu_1,\mu_2,\beta}(t)=\frac{a^{*}(1+t^2)}{\mu_1a^{2(p-1)}+\mu_2b^{2(p-1)}t^2+2\beta a^{p-1}b^{p-1}t}.
$$
Then
\begin{align}\label{of}
\mathcal{O}(\mu_1,\mu_2,\beta)\ge\inf_{t\in(0,\infty)}f_{\mu_1,\mu_2,\beta}(t).
\end{align}
Since $0<\mu_1a^{2(p-1)}<a^{*}$, $0<\mu_2b^{2(p-1)}<a^{*}$ and $\beta<\sqrt{(a^{*}-\mu_1a^{2(p-1)})(a^{*}-\mu_2b^{2(p-1)})}$,  $f_{\mu_1,\mu_2,\beta}(t)>1$ for $t\in(0,\infty)$ and
\begin{align*}
\lim_{t\to 0^{+}}f_{\mu_1,\mu_2,\beta}(t)=\frac{a^{*}}{\mu_1a^{2(p-1)}}>1,~\lim_{t\to \infty}f_{\mu_1,\mu_2,\beta}(t)=\frac{a^{*}}{\mu_2b^{2(p-1)}}>1.
\end{align*}
Hence  $\inf\limits_{t\in(0,\infty)}f_{\mu_1,\mu_2,\beta}(t)>1$ can be obtained from the continuity of $f_{\mu_1,\mu_2,\beta}(t)$. Therefore, from \eqref{of}, $\mathcal{O}(\mu_1,\mu_2,\beta)>1$. By Proposition \ref{propo1} $(i)$, \eqref{p3} has at least one minimizer. Hence $(i)$ holds.

$(ii)$ We take a function $0\le\phi_{b}\in C_{0}^{\infty}(\mathbb{R}^N)$ with $\|\phi_{b}\|_2=b$ and set
$$
u_{t}(x)=\frac{a}{\|Q\|_2}t^{\frac{N}{2s}}Q(t^{\frac{1}{s}}x)\in S_{a},~t>0,
$$
where $Q(x)$ satisfies \eqref{Qequation}. By standard computation and Remark \ref{rpi}, we have
\begin{align*}
\|(-\Delta)^{\frac{s}{2}}u_{t}\|_2^2=\frac{a^2t^2}{\|Q\|^2_2}\|(-\Delta)^{\frac{s}{2}}Q\|_2^2=\frac{a^2t^2N}{\alpha+2s},
\end{align*}
and
\begin{align*}
&B(u_{t},u_{t},p)=\frac{a^{2p}t^2}{\|Q\|^{2p}_2}B(Q,Q,p)=\frac{N+\alpha+2s}{\alpha+2s}\frac{a^{2p}t^2}{a^{*}},\\
&B(u_{t},\phi_{b},p)=\frac{a^{p}t^{\frac{N(\frac{p}{2}-1)-\alpha}{s}}}{\|Q\|^{p}_2}\int_{\mathbb{R}^N}\int_{\mathbb{R}^N}\frac{|Q(x)|^{p}|\phi_{b}(\frac{y}{t^{\frac{1}{s}}})|^{p}}{|x-y|^{N-\alpha}}dxdy=C\frac{a^{p}t^{\frac{N(\frac{p}{2}-1)-\alpha}{s}}}{\|Q\|^{p}_2},
\end{align*}
for some constant $C>0$.

For $\mu_1a^{2(p-1)}<a^{*}$, we take $(u_{t},\phi_{b})\in S_{a}\times S_{b}$ as a trial function for $\mathcal{O}(\mu_1,\mu_2,\beta)$, thus
\begin{align*}
\mathcal{O}(\mu_1,\mu_2,\beta)&\le\frac{A(u_{t},\phi_{b})}{\frac{\mu_1}{p}B(u_{t},u_{t},p)+\frac{\mu_2}{p}B(\phi_{b},\phi_{b},p)+\frac{2\beta}{p}B(u_{t},\phi_{b},p)}\\
&\le\frac{\frac{a^2t^2N}{\alpha+2s}+\|(-\Delta)^{\frac{s}{2}}\phi_{b}\|_2^2}{\frac{N\mu_1}{\alpha+2s}\frac{a^{2p}t^2}{a^{*}}+\frac{\mu_2}{p}B(\phi_{b},\phi_{b},p)+\frac{2\beta}{p}C\frac{a^{p}t^{\frac{N(\frac{p}{2}-1)-\alpha}{s}}}{\|Q\|^{p}_2}}\\
&\le\frac{\frac{a^2t^2N}{\alpha+2s}+\|(-\Delta)^{\frac{s}{2}}\phi_{b}\|_2^2}{\frac{N\mu_1}{\alpha+2s}\frac{a^{2p}t^2}{a^{*}}}\rightarrow\frac{a^{*}}{\mu_1a^{2(p-1)}}<1,\text{~as~}t\rightarrow\infty.
\end{align*}
Hence $\mathcal{O}(\mu_1,\mu_2,\beta)<1$. Combining with Proposition \ref{propo1} $(ii)$, \eqref{p3} has no minimizer.

Similarly, if $\mu_2b^{2(p-1)}<a^{*}$, we take $(\phi_{a}(x),v_{t}(x))\in S_{a}\times S_{b}$ as a trial function for $\mathcal{O}(\mu_1,\mu_2,\beta)$, where
\begin{align*}
&0\le\phi_{a}\in C_{0}^{\infty}(\mathbb{R}^N),~\|\phi_{a}\|_2=a\\
&v_{t}(x)=\frac{b}{\|Q\|_2}t^{\frac{N}{2s}}Q(t^{\frac{1}{s}}x)\in S_{b},~t>0.
\end{align*}
Then $\mathcal{O}(\mu_1,\mu_2,\beta)\le\frac{a^{*}}{\mu_2b^{2(p-1)}}<1$~as~$t\rightarrow\infty$.
Hence by Proposition \ref{propo1} $(ii)$, we can obtain the nonexistence of minimizer for \eqref{p3}.

Finally, if $\beta>\frac{(a^2+b^2)a^{*}-\mu_1a^{2p}+\mu_2b^{2p}}{2a^{p}b^{p}}$, we take $(u_{t}(x),v_{t}(x))\in S_{a}\times S_{b}$ as a trial function for $\mathcal{O}(\mu_1,\mu_2,\beta)$. Then
\begin{align*}
\mathcal{O}(\mu_1,\mu_2,\beta)&\le\frac{A(u_{t},v_{t})}{\frac{\mu_1}{p}B(u_{t},u_{t},p)+\frac{\mu_2}{p}B(v_{t},v_{t},p)+\frac{2\beta}{p}B(u_{t},v_{t},p)}\\
&=\frac{(a^2+b^2)a^{*}}{\mu_1a^{2p}+\mu_2b^{2p}+2\beta a^{p}b^{p}}<1.
\end{align*}
Hence \eqref{p3} has no minimizer by Proposition \ref{propo1} $(ii)$ again. In conclusion, we finish the proof.
\end{proof}

\begin{proof}[\textbf{Proof of Theorem \ref{t21}}] The proof is similar to the proof of Theorem \ref{t12}, hence we omit it.
\end{proof}
\section{$L^2$-supercritical}
In this section, we consider $1+\frac{\alpha+2s}{N}<p,q,r<\frac{N+\alpha}{N+2s}$.
We use the structure which is first introduced by L.Jeanjean in \cite{J97} and similar to the Section 4 in \cite{LHXY20}. Define the map
\begin{align*}
l\star(u,v)=(l\star u,l\star v):=(e^{\frac{Nl}{2s}}u(e^{\frac{l}{s}}x),e^{\frac{Nl}{2s}}v(e^{\frac{l}{s}}x)).
\end{align*}
It is easy to observe that $\|l\star u\|_2^2=\|u\|_2^2$ and $\|l\star v\|_2^2=\|v\|_2^2$.
\begin{lemma}\label{sinfty}
Suppose that $(u,v)\in S_{a}\times S_{b}$. Then
\begin{align*}
\lim_{l\rightarrow-\infty}A(l\star u,l\star v)=0,~\lim_{l\rightarrow\infty}A(l\star u,l\star v)=\infty,
\end{align*}
and
\begin{align*}
\lim_{l\rightarrow-\infty}E_{\mu_1,\mu_2,\beta}(l\star(u,v))=0^{+},~\lim_{l\rightarrow\infty}E_{\mu_1,\mu_2,\beta}(l\star(u,v))=-\infty.
\end{align*}
\end{lemma}
\begin{proof}
We can obtain the lemma by direct computation.
\end{proof}

From \eqref{gGN}, it follows that
\begin{align}\label{estimate u^p}
\begin{split}
&\frac{\mu_1}{2p}B(u,u,p)+\frac{\mu_2}{2q}B(v,v,q)+\frac{\beta}{r}B(u,v,r)\\
\le &C\|(-\Delta)^{\frac{s}{2}}u\|_2^{2p\delta_{p}}a^{2q(1-\delta_{q})}
+C\|(-\Delta)^{\frac{s}{2}}v\|_2^{2q\delta_{q}}b^{2q(1-\delta_{q})}\\
&+C\|(-\Delta)^{\frac{s}{2}}u\|_2^{r\delta_{r}}\|(-\Delta)^{\frac{s}{2}}v\|_2^{r\delta_{r}}(ab)^{r(1-\delta_{r})}\\
\le& \bar{C}\max\{A^{\delta}(u,v),A^{\eta}(u,v)\},
\end{split}
\end{align}
where $\delta=\min\{p\delta_{p},q\delta_{q},r\delta_{r}\}$ and $\eta=\max\{p\delta_{p},q\delta_{q},r\delta_{r}\}$.
\begin{lemma}\label{lmps}
There exist $K>0$ sufficient small such that for the sets
\begin{align*}
\begin{split}
&\Omega=\{(u,v)\in S_{a}\times S_{b}:A(u,v)\le K\}\\
&\Pi=\{(u,v)\in S_{a}\times S_{b}:A(u,v)= 2K\}
\end{split}
\end{align*}
there holds
\begin{itemize}
  \item [(i)] for any $(u,v)\in\Omega$, there holds that $E(u,v)>0$;
  \item [(ii)] $\sup\limits_{\Omega}E(u,v)<\inf\limits_{\Pi}E(u,v)$;
  \item[(iii)]$\inf\limits_{\Pi}E(u,v)>0$.
\end{itemize}
\end{lemma}
\begin{proof}
First, if $(u,v)\in\Omega$ (with $K$ to be determined), then by \eqref{estimate u^p}, we derive that
\begin{align*}
E_{\mu_1,\mu_2,\beta}(u,v):=&\frac{1}{2}A(u,v)-\frac{\mu_1}{2p}B(u,u,p)-\frac{\mu_2}{2q}B(v,v,q)-\frac{\beta}{\gamma}B(u,v,r)\\
\ge&\frac{1}{2}A(u,v)-\bar{C}\max\{A^{\delta}(u,v),A^{\eta}(u,v)\}>0,
\end{align*}
provided $K>0$ sufficient small due to $\delta>1$ and $\eta>1$. Therefore $(i)$ holds.

Next, we  show that $(ii)$ holds. Now if $(u_1,v_1)\in\Pi$ and $(u_2,v_2)\in\Omega$, making $K$ smaller if necessary, we have
\begin{align}\label{-}
\begin{split}
&E_{\mu_1,\mu_2,\beta}(u_1,v_1)-E_{\mu_1,\mu_2,\beta}(u_2,v_2)\\
\ge&\frac{2K}{2}-\bar{C}\max\{(2K)^{\delta},(2K)^{\eta}\}-\frac{K}{2}=\frac{K}{2}-\bar{C}\max\{(2K)^{\delta},(2K)^{\eta}\}>0.
\end{split}
\end{align}
Therefore, $E(u_1,v_1)-E(u_2,v_2)>0$ holds for any $(u_1,v_1)\in\Pi$ and $(u_2,v_2)\in\Omega$. Thus $(ii)$ holds.

Finally, similar to \eqref{-}, it follows that for any $(u_1,v_1)\in\Pi$
\begin{align*}
E_{\mu_1,\mu_2,\beta}(u_1,v_1)
\ge\frac{2K}{2}-\bar{C}\max\{(2K)^{\delta},(2K)^{\eta}\}>\frac{K}{2}>0.
\end{align*}
Hence $(iii)$ holds. In conclusion, we complete the proof.
\end{proof}
From now on, define the set
\begin{align*}
\Delta=\{(u,v)\in S_{a}\times S_{b}:A(u,v)\ge3K~\text{and}~E_{\mu_1,\mu_2,\beta}(u,v)\le 0\}.
\end{align*}
Using Lemma \ref{sinfty}, we can easily notice that $\Delta\neq\emptyset$. From Lemma \ref{sinfty} and Lemma \ref{lmps}, if we take $(\bar{u},\bar{v})\in\Omega$ and $(\hat{u},\hat{v})\in\Delta$, then there is a mountain path linking $(\bar{u},\bar{v})$ and $(\hat{u},\hat{v})$ and passing through $\Pi$. Let us define that
\begin{align*}
\Gamma=\{\gamma:=(\gamma_1(t),\gamma_2(t))\in C([0,1],S_{a}\times S_{b}):\gamma(0)=(\bar{u},\bar{v}),\gamma(1)=(\hat{u},\hat{v})\}.
\end{align*}
\begin{lemma}\label{lPs}
There exists a Palais-Smale sequence $(\bar{u}_n,\bar{v}_n)$ for $E$ on $S_{a}\times S_{b}$ at the level
\begin{align*}
c:=\inf_{\gamma\in\Gamma}\max_{t\in[0,1]}E_{\mu_1,\mu_2,\beta}(\gamma(t))\ge\inf_{\Pi}E>0\ge \max\{E_{\mu_1,\mu_2,\beta}(\bar{u},\bar{v}),E_{\mu_1,\mu_2,\beta}(\hat{u},\hat{v})\}
\end{align*}
satisfying additional condition
\begin{align}\label{add condition}
\begin{split}
&A(\bar{u}_n, \bar{v}_n)-\mu_1\delta_{p}B(\bar{u}_n, \bar{u}_n,p)-\mu_2\delta_{q}B(\bar{v}_n, \bar{v}_n,q)-2\beta\delta_{r}B(\bar{u}_n, \bar{v}_n,r)=o_n(1)
\end{split}
\end{align}
with $o_{n}(1)\rightarrow0$ as $n\rightarrow\infty$, where $\delta_{p}=\frac{N(p-1)-\alpha}{2p}$. Furthermore, $\bar{u}^{-}_n,\bar{v}^{-}_n\rightarrow0$ as $n\rightarrow\infty$.
\end{lemma}
\begin{proof}
We consider the auxiliary functional $\bar{E}_{\mu_1,\mu_2,\beta}$:
\begin{align*}
\begin{split}
&\bar{E}_{\mu_1,\mu_2,\beta}:\mathbb{R}\times S_{a}\times S_{b}\rightarrow\mathbb{R},\\
&\bar{E}_{\mu_1,\mu_2,\beta}(l,u,v)=E_{\mu_1,\mu_2,\beta}(l\star u,l\star v).
\end{split}
\end{align*}
Set $\bar{\Gamma}:=\{\bar{\gamma}(t)=(l(t),\gamma_1(t),\gamma_2(t))\in C([0,1],\mathbb{R}\times S_{a}\times S_{b}):\bar{\gamma}(0)=(0,\bar{u},\bar{v}),\bar{\gamma}(1)=(0,\hat{u},\hat{v})\}$, we want to employ the minimax principle for $\bar{E}_{\mu_1,\mu_2,\beta}$ on the minimax class $\bar{\Gamma}$ at the level
\begin{align*}
\bar{c}:=\inf_{\bar{\gamma}\in\bar{\Gamma}}\sup_{t\in[0,1]}\bar{E}_{\mu_1,\mu_2,\beta}(\bar{\gamma}(t)).
\end{align*}

We next show $c=\bar{c}$. On the one hand, thanks to $\Gamma\subset\bar{\Gamma}$, it is easy to find that $c\ge\bar{c}$. On the other hand, for any $\bar{\gamma}\in\bar{\Gamma}$, using the notation
\begin{align*}
\bar{\gamma}(t)=(l(t),\gamma_1(t),\gamma_2(t)),~t\in[0,1].
\end{align*}
By the definition of $\bar{E}$, it can be inferred that
\begin{align*}
\bar{E}_{\mu_1,\mu_2,\beta}(\bar{\gamma}(t))=E_{\mu_1,\mu_2,\beta}(l(t)\star\gamma_1(t),l(t)\star\gamma_2(t)),
\end{align*}
from $\bar{\gamma}\in\bar{\Gamma}$, it follows that $(l(t)\star\gamma_1(t),l(t)\star\gamma_2(t))\in\Gamma$. Hence  $c\le\bar{c}$. In conclusion, $c=\bar{c}$.

Notice that,
\begin{align*}
\bar{E}_{\mu_1,\mu_2,\beta}(\bar{\gamma}(t))=E_{\mu_1,\mu_2,\beta}(l(t)\star\gamma_1(t),l(t)\star\gamma_2(t))=\bar{E}(0,l(t)\star\gamma_1(t),l(t)\star\gamma_2(t)).
\end{align*}
Since  $\bar{E}_{\mu_1,\mu_2,\beta}(s,|u|,|v|)\le\bar{E}_{\mu_1,\mu_2,\beta}(s,u,v)$,  we can choose the minimizing sequence $\bar{\gamma}_{n}=(l_n,\gamma_{1n},\gamma_{2n})$ for $\bar{c}$ satisfying
\begin{align*}
\gamma_{1n},~\gamma_{2n}\ge0 \text{ a.e. in }\mathbb{R}^N.
\end{align*}
Using Theorem 3.2 in \cite{G93}, there exists a Palais-Smale sequence $(l_n,u_n,v_n)\in\mathbb{R}\times S_{a}\times S_{b}$ for $\bar{E}$ at level $\bar{c}$ such that
\begin{itemize}
  \item [(1)]$\lim\limits_{n\rightarrow\infty}\bar{E}_{\mu_1,\mu_2,\beta}(l_n,u_n,v_n)=\bar{c}=c$,
  \item[(2)]$\lim\limits_{n\rightarrow\infty}|l_n|+dist((u_n,v_n),(\gamma_{1n},\gamma_{2n}))=0$,
  \item [(3)]$\lim\limits_{n\rightarrow\infty}\|D_{\mathbb{R}\times S_{a}\times S_{b}}\bar{E}_{\mu_1,\mu_2,\beta}(l_n,u_n,v_n)\|=0$.
\end{itemize}
Thus $E_{\mu_1,\mu_2,\beta}(l_n\star u_n,l_n\star v_n)=c$ and $l_n\rightarrow 0$  from $(1)$ and $u_n,v_n\ge0$ from $\gamma_{1n},\gamma_{2n}\ge0$ and $(2)$. Therefore, $\{\bar{u}_n,\bar{v}_n\}=\{(l_n\star u_n,l_n\star v_n)\}$ is a Palais-Smale sequence for $E_{\mu_1,\mu_2,\beta}$. Similarly, for any $(\varphi_n,\psi_n)\in H_{r}^s\times H_{r}^s$, setting $(\tilde{\varphi}_n,\tilde{\psi}_n)=((-l_n)\star \varphi_n,(-l_n)\star \psi_n)$, it is easily to find that $\int_{\mathbb{R}^N}\bar{u}_n\varphi_n=0$ and $\int_{\mathbb{R}^N}\bar{v}_n\psi_n=0$ is equivalent to $\int_{\mathbb{R}^N}u_n\tilde{\varphi}_n=0$ and $\int_{\mathbb{R}^N}v_n\tilde{\psi}_n=0$. Moreover,
\begin{align*}
DE_{\mu_1,\mu_2,\beta}(\bar{u}_n,\bar{v}_n)[(\varphi_n,\psi_n)]=D\bar{E}_{\mu_1,\mu_2,\beta}(l_n,u_n,v_n)[(0,\tilde{\varphi}_n,\tilde{\psi}_n)]+o(1)\|(\tilde{\varphi}_n,\tilde{\psi}_n)\|_{H^{s}},
\end{align*}
due to $\|(\tilde{\varphi}_n,\tilde{\psi}_n)\|^2_{H^s(\mathbb{R}^N)\times H^s(\mathbb{R}^N)}\le 2\|(\varphi_n,\psi_n)\|^2_{H^s(\mathbb{R}^N)\times H^s(\mathbb{R}^N)}$ for $n$ large, it can be deduced that $\nabla_{S_{a}\times S_{b}}E(\bar{u}_n,\bar{v}_n)\rightarrow 0$. Using $(3)$, it implies that
\begin{align*}
D\bar{E}_{\mu_1,\mu_2,\beta}(l_n,u_n,v_n)[(1,0,0)]\rightarrow0, ~~l_n\rightarrow 0\text{~as~}n\rightarrow\infty.
\end{align*}
We can compute that
\begin{align*}
\left\{\begin{aligned}
&\partial_{l}(\frac{1}{2}A(l\star u,l\star v))=e^{2l}A(u,v),\\
&\partial_{l}\bigg\{\frac{\mu_1}{2p}B(l\star u,l\star u,p)+\frac{\mu_2}{2q}B(l\star v,l\star v,q)+\frac{\beta}{r}B(l\star u,l\star v,r)\bigg\}\\
&=\delta_{p}\mu_1e^{\frac{N(p-1)-\alpha}{2s}l}B(u,u,p)+\delta_{q}\mu_2e^{\frac{N(q-1)-\alpha}{2s}l}B(v,v,q)+2\delta_{r}\beta e^{\frac{N(r-1)-\alpha}{2s}l}B(u,v,r),
\end{aligned}
\right.
\end{align*}
from $(3)$, it follows that
\begin{align*}
&A(\bar{u}_n,\bar{v}_n)-\delta_{p}\mu_1B(\bar{u}_n,\bar{u}_n,p)-\delta_{q}\mu_2B(\bar{v}_n,\bar{v}_n,q)-2\delta_{r}\beta B(\bar{u}_n,\bar{v}_n,r)\rightarrow 0~~\text{as~}n\rightarrow\infty,
\end{align*}
i.e. $(\bar{u}_n,\bar{v}_n)$ satisfies the additional condition. Hence we finish the proof.
\end{proof}

\begin{lemma}\label{lbdd}
Suppose $(\bar{u}_n,\bar{v}_n)$ is obtained in Lemma \ref{lPs}. Then $(\bar{u}_n,\bar{v}_n)$ is bounded in $H^s(\mathbb{R}^N)\times H^s(\mathbb{R}^N)$.
\end{lemma}
\begin{proof}
By the additional condition, it yields that
\begin{align*}
E_{\mu_1,\mu_2,\beta}(\bar{u}_n,\bar{v}_n)
=&\frac{1}{2}A(\bar{u}_n,\bar{v}_n)-\frac{\mu_1}{2p}B(\bar{u}_n,\bar{u}_n,p)-\frac{\mu_2}{2q}B(\bar{v}_n,\bar{v}_n,q)-\frac{\beta}{r}B(\bar{u}_n,\bar{v}_n,r)\\
=&\frac{\mu_1}{2}\frac{N(p-1)-\alpha-2s}{2ps}B(\bar{u}_n,\bar{u}_n,p)+\frac{\mu_2}{2}\frac{N(q-1)-\alpha-2s}{2qs}B(\bar{v}_n,\bar{v}_n,q)\\
&+\beta\frac{N(r-1)-\alpha-2s}{2ps}B(\bar{u}_n,\bar{v}_n,r)\text{~as~} n\rightarrow\infty.
\end{align*}
Since $1+\frac{\alpha+2s}{N}<p,q,r<\frac{N+\alpha}{N-2s}$, the coefficients of each terms in above equality are positive. By the fact $E_{\mu_1,\mu_2,\beta}(\bar{u}_n,\bar{v}_n)\rightarrow c$ as $n\rightarrow\infty$, it follows that  there exist $C>0$  and $\bar{C}$ such that
\begin{align}\label{estimateB}
 \bar{C}\le B(\bar{u}_n,\bar{u}_n,p)+B(\bar{u}_n,\bar{v}_n,r)+B(\bar{v}_n,\bar{v}_n,q)\le C.
\end{align}
Applying the fact $E_{\mu_1,\mu_2,\beta}(\bar{u}_n,\bar{v}_n)\rightarrow c$ as $n\rightarrow\infty$ again, we can obtain that $A(\bar{u}_n,\bar{v}_n)\le C$.
Hence the desired results is obtained.
\end{proof}

Since $E_{\mu_1,\mu_2,\beta}'|_{S_{a}\times S_{b}}(\bar{u}_n,\bar{v}_n)\rightarrow 0$, there exist two sequence of real numbers $\{\lambda_{1n}\}$ and $\{\lambda_{2n}\}$ such that
\begin{align*}
E_{\mu_1,\mu_2,\beta}'(\bar{u}_n,\bar{v}_n)(\varphi,\psi)=o_n(1),~~\forall(\varphi,\psi)\in H^s(\mathbb{R}^N)\times H^s(\mathbb{R}^N),
\end{align*}
i.e,
\begin{align}\label{dvt e}
\begin{split}
&\int_{\mathbb{R}^N}(-\Delta)^{\frac{s}{2}} \bar{u}_n(-\Delta)^{\frac{s}{2}} \varphi dx
-\mu_1\int_{\mathbb{R}^N}(I_{\alpha}\star|\bar{u}_n|^p)|\bar{u}_n|^{p-2}\bar{u}_ndx\\
&-\int_{\mathbb{R}^N}\beta(x)((I_{\alpha}\star|\bar{v}_n|^r)|\bar{u}_n|^{r-2}\bar{u}_n\varphi dx+\lambda_{1n}\int_{\mathbb{R}^N}\bar{u}_n\varphi dx\\
&+\int_{\mathbb{R}^N}(-\Delta)^{\frac{s}{2}} \bar{v}_n(-\Delta)^{\frac{s}{2}} \psi dx-\mu_2\int_{\mathbb{R}^N}(I_{\alpha}\star|\bar{v}_n|^p)|\bar{v}_n|^{p-2}\bar{v}_n\psi dx\\
&-\int_{\mathbb{R}^N}\beta(x)(I_{\alpha}\star|\bar{u}_n|^r)|\bar{v}_n|^{r-2}\bar{v}_n\psi dx
+\lambda_{2n}\int_{\mathbb{R}^N}\bar{v}_n\psi dx
=o_{n}(1),
\end{split}
\end{align}
for every $(\varphi,\psi)\in H^s(\mathbb{R}^N)\times H^s(\mathbb{R}^N)$ with $o_{n}(1)\rightarrow 0$ as $n\rightarrow\infty$.

\begin{lemma}\label{llcp}
Both $\{\lambda_{1n}\}$ and $\{\lambda_{2n}\}$ are bounded sequences. In addition, up to a subsequence, at least one of the sequences converges to a strict positive value.
\end{lemma}
\begin{proof}
Taking $(\varphi,\psi)$ to be $(\bar{u}_n,0)$ and $(0,\bar{v}_n)$ respectively in \eqref{dvt e}, it follows that
\begin{align}\label{l1}
-\lambda_{1n}a^2+o_{n}(1)
=\|(-\Delta)^{\frac{s}{2}} \bar{u}_n\|_2^2-\mu_1B(\bar{u}_n,\bar{u}_n,p)-\beta B(\bar{u}_n,\bar{v}_n,r),
\end{align}
and
\begin{align}\label{l2}
-\lambda_{2n}b^2+o_{n}(1)
=\|(-\Delta)^{\frac{s}{2}} \bar{v}_n\|_2^2-\mu_2B(\bar{v}_n,\bar{v}_n,q)-\beta B(\bar{u}_n,\bar{v}_n,r).
\end{align}
Thanks to the boundedness of  $(\bar{u}_n,\bar{v}_n)$ in $H^s_{r}\times H^s_{r}$, we can deduce that $\{\lambda_{1n}\}$ and $\{\lambda_{2n}\}$ are bounded. Moreover, by \eqref{add condition}, \eqref{l1} and \eqref{l2}, there holds
\begin{align*}
&\lambda_{1n}a^2+\lambda_{2n}b^2+o_{n}(1)\\
=&-A(\bar{u}_n,\bar{v}_n)+\mu_1B(\bar{u}_n,\bar{u}_n,p)+\mu_2B(\bar{v}_n,\bar{v}_n,q)+2\beta B(\bar{u}_n,\bar{v}_n,r)\\
=&\mu_1(1-\delta_{p})B(\bar{u}_n,\bar{u}_n,p)+\mu_2(1-\delta_{q})B(\bar{v}_n,\bar{v}_n,q)+2\beta (1-\delta_{r})B(\bar{u}_n,\bar{v}_n,r).
\end{align*}
Since $\delta_{p},\delta_{q},\delta_{r}<1$, by \eqref{estimateB}, it can be inferred that there exists $\tilde{C}>0$ such that
\begin{align*}
\lambda_{1n}a^2+\lambda_{2n}b^2\ge\tilde{C}>0,
\end{align*}
for $n$ large.  Hence the lemma holds.
\end{proof}
From the above lemma, there exist $\lambda_1\in\mathbb{R}$ and $\lambda_2\in\mathbb{R}$ such that  $\lambda_{1n}\rightarrow\lambda_1$ and $\lambda_{2n}\rightarrow\lambda_2$ as $n\rightarrow\infty$ up to a subsequence. The signs of $\lambda_1$ and $\lambda_2$ play an important role in the strong convergence in $H^{s}(\mathbb{R}^N)$.
\begin{lemma}\label{llp}
If $\lambda_1>0$ (resp. $\lambda_2>0$), then $\bar{u}_n\rightarrow\bar{u}$ (resp. $\bar{v}_n\rightarrow\bar{v}$) strongly in $H^{s}(\mathbb{R}^N)$.
\end{lemma}
\begin{proof}
By Lemma \ref{lbdd}, $(\bar{u}_n,\bar{v}_n)$ is bounded in $H^{s}_{r}(\mathbb{R}^N)\times H^{s}_{r}(\mathbb{R}^N)$. Since $H^{s}_{r}(\mathbb{R}^N)\hookrightarrow L^{p}(\mathbb{R}^N)$ for $p\in(2,2^{*}_{s})$  is compact, by \eqref{gGN}, we have
\begin{align}\label{chlp}
\left\{\begin{aligned}
&(\bar{u}_n,\bar{v}_n)\rightharpoonup(\bar{u},\bar{v})\text{~weakly in }H^s_{r}(\mathbb{R}^N)\times H^s_{r}(\mathbb{R}^N),\\
&B(\bar{u}_n,\bar{u}_n,p)\to B(\bar{u},\bar{u},p),B(\bar{u}_n,\bar{v}_n,r)\to B(\bar{u},\bar{v},r),B(\bar{v}_n,\bar{v}_n,q)\to B(\bar{v},\bar{v},q)
\\
&(\bar{u}_n,\bar{v}_n)\to(\bar{u},\bar{v})\text{~a.e. in }\mathbb{R}^N.
\end{aligned}
\right.
\end{align}
From \eqref{dvt e}, it yields that
\begin{align*}
\langle E_{\mu_1,\mu_2,\beta}'(\bar{u}_n,\bar{v}_n)-\lambda_{1n}(\bar{u}_n,0),(\bar{u}_n,0)\rangle\rightarrow\langle E_{\mu_1,\mu_2,\beta}'(\bar{u},\bar{v})-\lambda_{1}(\bar{u},0),(\bar{u},0)\rangle~~~\text{as }n\rightarrow\infty.
\end{align*}
Together with \eqref{chlp}, we can obtain that
\begin{align*}
\|(-\Delta)^{\frac{s}{2}} \bar{u}_n\|_2^2+\lambda_{1n}\|\bar{u}_n\|_2^2\rightarrow\|(-\Delta)^{\frac{s}{2}} \bar{u}\|_2^2+\lambda_{1}\|\bar{u}\|_2^2.
\end{align*}
From $\lambda_1>0$, it means that  $\bar{u}_n\rightarrow\bar{u}$ strongly in $H^{s}(\mathbb{R}^N)$. For $\lambda_2>0$, the proof is similar. Hence we omit the details. Therefore, we complete the proof.
\end{proof}

\begin{lemma}\label{leub}
There exists $\beta_1>0$ such that if $\beta>\beta_1$, then
\begin{align}\label{ec}
0<c<min\{n(a,\mu_1,p),n(b,\mu_2,q)\}.
\end{align}
\end{lemma}

\begin{proof}
It is obvious to check that $c>0$. We only show the right side of \eqref{ec} holds. Denote that $u_0$ is the ground state solution of \eqref{scalar problem} with $\mu=\mu_1,p=p,c=a$ and $v_0$ is the ground state solution of \eqref{scalar problem} with $\mu=\mu_2,p=q,c=b$.
By Lemma \ref{sinfty}, it follows that there exists sufficient small $l_0<0$ and  large enough $l_1>0$ such that
\begin{align*}
l_0\star(u_0,v_0)\in\Omega\text{~and~}l_1\star(u_0,v_0)\in\Delta.
\end{align*}
Taking the path $\gamma_0(t)=[(1-t)l_0+tl_1]\star(u_0,v_0)(t\in[0,1])$, it is easy to check that $\gamma_0\in\Gamma$. By the definition of $c$,  we have
\begin{align*}
c\le \max_{t\in[0,1]}E_{\mu_1,\mu_2,\beta}(\gamma_0(t))\le\sup_{l\in\mathbb{R}}E_{\mu_1,\mu_2,\beta}(l\star(u,v)).
\end{align*}

Now we show that $\sup\limits_{l\in\mathbb{R}}E_{\mu_1,\mu_2,\beta}(l\star(u,v))<min\{n(a,\mu_1,p),n(b,\mu_2,q)\}$.
Since $\beta>0$, we can compute that
\begin{align}\label{ee}
\begin{split}
&E_{\mu_1,\mu_2,\beta}(l\star(u_0,v_0))\\
=&\frac12A(l\star u_0,l\star v_0)-\frac{\mu_1}{2p}B(l\star u_0,l\star u_0,p)\\
&-\frac{\mu_2}{2q}B(l\star v_0,l\star v_0,q)-\frac{\beta}{r}B(l\star u_0,l\star v_0,r)\\
\le&\frac12A(l\star u_0,l\star v_0)-\frac{\mu_1}{2p}B(l\star u_0,l\star u_0,p)-\frac{\mu_2}{2q}B(l\star v_0,l\star v_0,q)
\end{split}
\end{align}
In addition, we notice that
\begin{align*}
M_{\mu_1}^{p}(l\star u_0)\rightarrow0,~M_{\mu_2}^{q}(l\star v_0)\rightarrow0,\text{~as~}l\rightarrow-\infty,
\end{align*}
combining with \eqref{ee}, there exists sufficient small $\bar{l}<0$ such that
\begin{align*}
\sup_{l\le \bar{l}}E_{\mu_1,\mu_2,\beta}(l\star(u_0,v_0))
\le&\sup_{l\le \bar{l}}\{M_{\mu_1}^{p}(l\star u_0)+M_{\mu_2}^{q}(l\star u_0)\}\\
<&\min\{n(a,\mu_1,p),n(b,\mu_2,q)\}.
\end{align*}
On the other hand, for $l\ge\bar{l}$, we have
\begin{align*}
B(l\star u_0,l\star v_0,r)=e^{l\frac{N(p-1)-\alpha}{s}}B(u_0,v_0,r)\ge Ce^{\bar{l}\frac{N(p-1)-\alpha}{s}},
\end{align*}
together with \eqref{ee}, we can derive that
\begin{align*}
\sup_{l\ge \bar{l}}E_{\mu_1,\mu_2,\beta}(l\star(u_0,v_0))
\le&\sup_{l\ge \bar{l}}\{M_{\mu_1}^{p}(l\star u_0)+M_{\mu_2}^{q}(l\star u_0)\}-Ce^{\bar{l}\frac{N(p-1)-\alpha}{s}}\\
\le&n(a,\mu_1,p)+n(b,\mu_2,q)-Ce^{\bar{l}\frac{N(p-1)-\alpha}{s}}\\
<&\min\{n(a,\mu_1,p),n(b,\mu_2,q)\},
\end{align*}
provided $\beta$ sufficient large. As a consequence, there exists $\beta_1>0$ sufficient large such that $c<min\{n(a,\mu_1,p),n(b,\mu_2,q)\}$ for $\beta>\beta_1$. Therefore, we finish the proof.
\end{proof}
\begin{proof}[\textbf{Proof of Theorem \ref{t2}}]
From Lemma \ref{llcp}, without loss of generality, we may suppose $\lambda_1>0$. Thus $\bar{u}_n\rightarrow\bar{u}$ strongly in $H^s_r(\mathbb{R}^N)$ by Lemma \ref{llp}. We only remain to show that $\bar{v}_n\rightarrow\bar{v}$ strongly in $H^s_r(\mathbb{R}^N)$. We argue it by contradiction and assume that $\lambda_2\le0$. By the weak limit and \eqref{dvt e}, $\bar{v}$ satisfies
\begin{align*}
(-\Delta)^{s}\bar{v}=-\lambda_2 \bar{v}+\mu_2(I_{\alpha}\star |v|^p)|\bar{v}|^{q-2}\bar{v}+\beta(I_{\alpha}\star |u|^p)|\bar{v}|^{q-2}\bar{v} \ge0.
\end{align*}
By Lemma \ref{llt}, we have $\bar{v}\equiv0$. Hence $E_{\mu_1,\mu_2,\beta}(\bar{u},\bar{v})=M_{\mu_1}^{p}(\bar u)=n(a,\mu_1,p)$, which contradicts to Lemma \ref{leub}. Therefore, $\lambda_2>0$, which implies $\bar{v}_n\rightarrow\bar{v}$ strongly in $H^s_{r}(\mathbb{R}^N)$.
\end{proof}

\end{document}